\documentclass[11pt,reqno]{amsart}

\usepackage{titlesec} 
\titleformat{\section}{\vskip10pt\normalsize\bfseries}{\thesection.}{0.5em}{\centering}
\titleformat{\subsection}{\vskip10pt\normalsize\bfseries}{\thesubsection.}{0.5em}{}

\usepackage[margin=1.1in]{geometry}
\allowdisplaybreaks

\usepackage{eucal,color,graphicx,subfigure,mathrsfs}

\usepackage{stmaryrd}
\usepackage{hyperref}
\usepackage{multicol}
\usepackage{multirow}
\usepackage{booktabs}
\usepackage{subfigure}
\numberwithin{equation}{section}

\newtheorem{theorem}{Theorem}[section]
\newtheorem{lemma}[theorem]{Lemma}

\theoremstyle{definition}

\theoremstyle{remark}
\newtheorem{remark}{\bf Remark}
\numberwithin{remark}{section}

\numberwithin{equation}{section}

\renewcommand{\d}{{\mathrm{d}}}
\newcommand{\R}{{\mathbb{R}}}

\newcommand{\Vn}{{\nu}}

\newcommand{\Vu}{{u}}

\newcommand{\CF}{{\mathcal F}}

\newcommand{\CT}{{\mathcal T}}

\newcommand{\lb}{\llbracket}
\newcommand{\rb}{\rrbracket}
\newcommand{\Lb}{\{\hspace{-4.0pt}\{}
\newcommand{\Rb}{\}\hspace{-4.0pt}\}}

\begin{document}

\title[Navier--Stokes equations with variable density]{A convergent post-processed discontinuous Galerkin method 
for incompressible flow with variable density}

\author{Buyang Li\,\,}
\address{Department of Applied Mathematics, The Hong Kong Polytechnic University, Hung Hom, Hong Kong.} 
\email{buyang.li@polyu.edu.hk}

\author{\,\,Weifeng Qiu}
\address{Department of Mathematics, City University of Hong Kong,
83 Tat Chee Avenue, Kowloon, Hong Kong.}
\email{weifeqiu@cityu.edu.hk}

\author{Zongze Yang\,\,}
\address{Department of Applied Mathematics, The Hong Kong Polytechnic University, Hung Hom, Hong Kong.} 
\email{zongze.yang@polyu.edu.hk}

\begin{abstract}
We propose a linearized semi-implicit and decoupled finite element method for the incompressible Navier--Stokes 
equations with variable density.  Our method is fully discrete and shown to be unconditionally stable. 
The velocity equation is solved by an $H^1$-conforming finite element method, and an upwind discontinuous Galerkin finite element method with post-processed velocity is adopted for the density equation. The proposed method is proved to be convergent in approximating reasonably smooth solutions in three-dimensional convex polyhedral domains. 
\end{abstract}


\maketitle 

\section{Introduction}

In this article we consider numerical approximation to incompressible flow with variable density, described by 
the following hyperbolic-parabolic system of partial differential equations (PDEs): 
\begin{subequations}
\label{DNS_eqs}
\begin{align}
\label{DNS_eq1}
\partial_{t} \rho + \nabla\cdot (\rho u) = 0
&&\mbox{in}\,\,\,\Omega\times(0,T], \\
\label{DNS_eq2}
\rho \partial_{t} u + \rho  ( u \cdot \nabla) u + \nabla p - \mu  \Delta  u = 0
&&\mbox{in}\,\,\,\Omega\times(0,T], \\ 
\label{DNS_eq3}
\nabla\cdot u = 0
&&\mbox{in}\,\,\,\Omega\times(0,T],
\end{align}
\end{subequations}
in a convex polyhedral domain $\Omega \subset \R^{d}$, with $d\in\{2,3\}$, up to a given time $T$, with the following boundary and initial conditions: 
\begin{subequations}
\label{DNS_IBC}
\begin{align}
\label{DNS_IBC_c1}
& u = 0  && \text{on } \partial\Omega \times [0, T],\\
\label{DNS_IBC_c2}
& \rho = \rho^{0} \text{ and } u = u^{0} 
&&\mbox{in}\,\,\,\Omega\,\,\, \text{ at } t = 0 .
\end{align}
\end{subequations}
In this model, $\rho:\Omega\rightarrow\R$, $u:\Omega\rightarrow\R^d$ and $p:\Omega\rightarrow\R$ are the density, velocity and pressure of the fluid, respectively, and $\mu > 0$ is the viscosity constant of the fluid. The initial value of the density is assumed to satisfy the following physical condition: 
\begin{align} 
\label{positive_initial_density}
\rho_{\rm min}
:= \min_{x \in \Omega} \rho^{0}(x) > 0 . 
\end{align} 
For smooth initial values satisfying the positivity condition \eqref{positive_initial_density}, existence and uniqueness of smooth solutions of \eqref{DNS_eqs} in two dimensions were proved in \cite{Danchin2006,LadyzhenskayaSolonnikov1978,Torres-Silva-Medar-2009}. Hence, this problem does not generate shock waves in finite time (at least in 2D). Existence and uniqueness of smooth solutions in three dimensions remains open similarly as the Navier--Stokes equations with constant density. 

Numerical approximation to the coupled system \eqref{DNS_eqs} were studied with many different numerical methods, including projection methods \cite{Almgren-Bell-Colella-Howell-Welcome-1998,Bell-Marcus-1992,Guermond-Quartapelle-2000,Li-Mei-Ge-Shi-2013,Pyo-Shen-2007}, fractional-step methods \cite{Guermond-Salgado-2008,Guermond-Salgado-2009}, backward differentiation formulae  \cite{Li-Li-Mei-Li-2015}, and the discontinuous Galerkin (DG) method \cite{Liu-Walkington-2007}. The stability of several numerical methods was proved in \cite{Guermond-Quartapelle-2000,Pyo-Shen-2007,Li-Mei-Ge-Shi-2013}. Convergence of a DG method and a staggered non-conforming finite element method were proved based on compactness arguments in \cite{Liu-Walkington-2007} and \cite{Latche-Saleh-2020}, respectively, without explicit convergence rates. 

Since the variable density introduces considerable difficulties to error analysis of the coupled nonlinear system, as mentioned in \cite{Pyo-Shen-2007}, error analysis has been done only in a few articles. The main difficulty is to prove boundedness of numerical solutions to both $\rho$ and $u$, as well as a positive lower bound of the numerical solution to $\rho$, uniformly with respect to the temporal stepsize and spatial mesh size. An error estimate for the single velocity equations \eqref{DNS_eq2} was presented in \cite{Guermond-Salgado-2011} for the methods proposed in \cite{Guermond-Salgado-2008,Guermond-Salgado-2009}, where the numerical solutions $\rho_h^n$, $n=1,\dots,N$, of the density equation were assumed to have positive upper and lower bound uniformly with respect to the temporal stepsize and spatial mesh size; see \cite[Conjectures in Remark 4.2]{Guermond-Salgado-2011}. An error estimate for a fractional-step temporally semidiscrete method was presented in \cite{An-2020-JSC} under the assumption that the numerical solution of density has positive upper and lower bounds uniformly with respect to the temporal stepsize.
The first complete error estimate of fully discrete FEM for the coupled system \eqref{DNS_eqs} was presented in 
\cite{Cai-Li-Li-2020} for the two-dimensional problem based on $H^3$ regularity assumption on the solution. 
The analysis in \cite{Cai-Li-Li-2020} utilizes an error splitting approach, which involves  analyzing the error of full discretization based on uniform regularity estimates for the temporally semidiscrete solutions. 
However, the analysis in \cite{Cai-Li-Li-2020} cannot be directly extended to three dimensions due to the presence of $H^1$-conforming finite element solution of $u$ in the density equation, which requires proving $W^{1,\infty}$-boundedness of the numerical solution to $u$ in order to obtain an error estimate for the density equation. This limits the analysis in \cite{Cai-Li-Li-2020} to two dimensions and solutions with $H^3$ regularity. 
Hence, error estimates for the three-dimensional problem based only on $H^{2+\alpha}$ spatial regularity of solutions (more realistic in general convex polyhedra) still remain open.

The objective of this article is to introduce a fully discrete, linearized semi-implicit, decoupled and unconditionally stable FEM for the coupled system \eqref{DNS_eqs}--\eqref{DNS_IBC} such that error analysis can be done in three dimensions under more realistic $H^{2+\alpha}$ regularity assumptions on the solution in a convex polyhedron.  
To this end, we propose an upwind DG method for the density equation with post-processed velocity, and $H^1$-conforming FEM for the velocity equation. The key to error analysis in three dimensions is the post-processing of velocity, which projects the $H^1$-conforming finite element solution of $u$ to the divergence-free subspace of the Raviart--Thomas element space. This post-processing has a significant influence on the error analysis: it allows us to derive an error estimate without proving the $W^{1,\infty}$-boundedness of the numerical solution to $u$. 

In Section \ref{sec:main_results}, we present the main results of this paper, including the numerical method and error estimate. 
The proof of the main theorem is presented in Section \ref{sec:proof}.

\section{Main results}\label{sec:main_results}

\subsection{Notation}

Let $\Omega$ be convex polygon/polyhedron in $\R^d$,  and denote by $\nu$ the outward unit normal vector 
on the boundary $\partial\Omega$.  We define the following function spaces on $\Omega$:
\begin{align}
&H^1(\Omega) :=\{v\in L^2(\Omega): \nabla v \in L^2(\Omega)^d\} ,\\
&\mathring H^1(\Omega) :=\{v\in H^1(\Omega): v=0\,\,\,\mbox{on}\,\,\,\partial\Omega \} ,\\
&\widetilde L^2(\Omega) :=\{v\in L^2(\Omega): \mbox{$\int_\Omega$} v \, \d x=0 \} ,\\
&H(\text{div}, \Omega) :=\{v\in L^2(\Omega)^d:\nabla\cdot v\in L^2(\Omega)\} .
\end{align}

For any nonnegative integer $r$, we denote by ${\rm P}^r_{\rm dG}(\CT_{h})$ the scalar-valued discontinuous 
Galerkin finite element space of degree up to $r$, built on a quasi-uniform partition $\CT_{h}$ of $\Omega$ into 
tetrahedra (with $\CT_{h}$ denoting the set of tetrahedra, and $h$ denoting the mesh size). 
The outward unit normal vector on the boundary $\partial K$ of a tetrahedron $K\in\CT_{h}$ is denoted by $\nu_K$. 

We define ${\rm RT}^1(\CT_{h})$ to be the $H(\text{div}, \Omega)$-conforming Raviart--Thomas 
finite element spaces of order $1$, i.e., 
$${\rm RT}^1(\CT_{h}) := \{w \in H(\text{div}, \Omega): w|_{K} \in P_{1}(K)^{d} + xP_{1}(K), 
\forall K \in \CT_{h} \} .
$$ 
We also define the following finite element spaces: 
\begin{align}
&{\rm P}^1(\CT_{h}) := {\rm P}^1_{\rm dG}(\CT_{h})\cap H^1(\Omega) , 
\\
&{\rm P}^{\rm 1b}(\CT_{h}) := {\rm P}^1(\CT_{h})\,\,\mbox{enriched by a bubble function (cf. \cite{Arnold-Brezzi-Fortin-1984} and \cite[Section 7.1]{Boffi-2008})}, 
\\
&\mathring{\rm P}^{\rm 1b}(\CT_{h}) := {\rm P}^{\rm 1b}(\CT_{h})\,\,\mbox{with with zero boundary condition}, 
\label{FEM_space_velocity} \\ 
&\widetilde{\rm P}^1(\CT_{h}) 
:= \{ v\in {\rm P}^1(\CT_{h}) : \mbox{$\int_\Omega v \d x=0$} \} , 
\label{FEM_space_pressure} \\
&{\rm RT}^1_0(\CT_{h})
:=\{v_h\in {\rm RT}^1(\CT_{h}):
\mbox{$\nabla\cdot v_h = 0$ in $\Omega$ and $v_h\cdot \nu=0$ on $\partial\Omega$}\} .
\label{FEM_space_post_processed_velocity}
\end{align}
We denote by $P_h^{\rm RT}:L^2(\Omega)^d \rightarrow {\rm RT}^1_0(\CT_{h})$ the $L^{2}$-orthogonal projection, defined by 
\begin{align}\label{L2-Proj-RT}
\,\,\,(v-P_h^{\rm RT}v,w_h)=0
\quad\forall\, w_h\in {\rm RT}^1_0(\CT_{h}),\,\,\,\forall\,v\in L^2(\Omega)^d . 
\end{align}
Similarly, we denote by $P_h^{\rm dG}:L^2(\Omega)\rightarrow {\rm P}^2_{\rm dG}(\CT_{h})$ 
the $L^{2}$-orthogonal projection defined by 
\begin{align}\label{L2-Proj-dG}
\begin{aligned}
&(v-P_h^{\rm dG}v,w_h)=0
&&\forall\, w_h\in {\rm P}^2_{\rm dG}(\CT_{h}),\,\,\,\forall\,v\in L^2(\Omega) .
\end{aligned}
\end{align}
The finite element space $\mathring{\rm P}^{\rm 1b}(\CT_{h})^d \times \widetilde{\rm P}^1(\CT_{h})$ satisfies the inf-sup condition (cf. \cite{Arnold-Brezzi-Fortin-1984,Boffi-2008}) 
\begin{align}
\|q_h\|_{L^2(\Omega)}\le C\sup_{\begin{subarray}{ll}
v_h\in \mathring{\rm P}^{\rm 1b}(\CT_{h})\\
v_h\neq 0
\end{subarray}}\frac{|(\nabla\cdot v_h,q_h)|}{\|v_h\|_{H^1(\Omega)}},\quad
\forall~q_h\in \widetilde{\rm P}^1(\CT_{h}) ,
\end{align}
and therefore is stable in approximating the Stokes and Navier--Stokes equations. 
This inf-sup condition is required in practical computation for the numerical method to be stable, but is not used in our error analysis. 

We denote by 
$$
(u,v)
=\sum_{K\in\CT_h} \int_K uv \d x
\quad
\langle u,v \rangle_{\partial K_s}
= \int_{\partial K_s} uv \,\d s
$$
the inner product of $L^2(\Omega)$ and $L^2(\partial K_s)$, respectively, where $\partial K_s$ is a subset of $\partial K$ for a tetrahedron $K\in\CT_h$. 
For a function $v$ uniformly continuous on each tetrahedron $K\in\CT_h$, we define 
\begin{align}\label{average-jump}
\Lb v \Rb
=\frac12(v^+ + v^-) 
\quad\mbox{and}\quad
\lb v \rb
=(v^- - v^+)\nu_K
\end{align} 
to be the average and jump of the function $v$ defined on the boundary $\partial K$ for $K\in\CT_h$, with $v^+$ and $v^-$ denoting the exterior and interior traces of the function. If $F=K\cap K'$ is a common face of two tetrahedra $K$ and $K'$, then the jump $\lb v\rb$ on $F$ is independent of the definitions using $K$ and $K'$. 

To guarantee the positivity of the numerical solution of density $\rho$, we denote by $\chi \in W^{1,\infty}(\mathbb{R})$ the cut-off function defined by 
$$
{\chi}(s) 
= \left\{
\begin{array}{cl}
\dfrac{1}{2}\rho_{\rm min} & \displaystyle\mbox{if}~~s <\frac{1}{2}\rho_{\rm min},  \\[10pt]
s &\displaystyle\mbox{if}~~ \frac{1}{2}\rho_{\rm min} \le s\le \frac{3}{2}\rho_{\rm max} ,\\[10pt]
\dfrac{3}{2}\rho_{\rm max}  &\displaystyle \mbox{if}~~s>\frac{3}{2}\rho_{\rm max}, 
\end{array}\right.
$$
where
\begin{align}\label{def-rho-min-max}
\rho_{\rm min}
:= \min_{x \in \Omega} \rho^{0}(x)
\quad\mbox{and}\quad
\rho_{\rm max}
:= \max_{x \in \Omega} \rho^{0}(x) . 
\end{align}
The cut-off function defined above has the following conditions: 
\begin{subequations}
\label{chi_properties}
\begin{align}
\label{chi_property1}
&\chi(s) = s  
&&\hspace{-50pt} \forall s \in \Big[\frac{1}{2}\rho_{\rm min} \, , \, 
\frac{3}{2}\rho_{\rm max} \Big] ,\\
\label{chi_property2}
&\dfrac{1}{2} \rho_{\rm min} \le \chi (s) \le \dfrac{3}{2} 
\rho_{\rm max}  
&& \hspace{-50pt} \forall s \in \mathbb{R} .
\end{align}
\end{subequations}

\subsection{The numerical method and its convergence}

Let $t_n=n\tau$, $n=0,1,\dots,N$, be a uniform partition of the time interval $[0,T]$ with stepsize $\tau=T/N$. 
For a given function $u_{h}^{n-1}$ at time $t=t_{n-1}$, we denote by $\partial K_{-}^{n}$ ($ \partial K_{+}^{n}$) 
the numerical inflow (outflow)  boundary of the tetrahedron $K\in\CT_h$ at time $t=t_n$, defined by 
$$
\partial K_{-}^{n} := \{x \in \partial K :  ( P_h^{\rm RT} u_{h}^{n-1} \cdot \nu_{K} ) (x)< 0\}, \quad 
 \partial K_{+}^{n} := \{x \in \partial K :  ( P_h^{\rm RT} u_{h}^{n-1} \cdot \nu_{K} ) (x)> 0\}.
$$ 
We consider the following fully discrete linearized FEM for \eqref{DNS_eqs}--\eqref{DNS_IBC} (based on a reformulation of the system as shown in \cite[(1.5)--(1.7)]{Cai-Li-Li-2020}): for given 
$(\rho_{h}^{n-1},u_{h}^{n-1}) \in {\rm P}^2_{\rm dG}(\CT_{h}) \times \mathring{\rm P}^{\rm 1b}(\CT_{h})^d $, 
find $(\rho_{h}^{n}, u_{h}^{n}, p_{h}^{n}) 
\in {\rm P}^2_{\rm dG}(\CT_{h}) \times \mathring{\rm P}^{\rm 1b}(\CT_{h})^d \times \widetilde{\rm P}^1(\CT_{h})$ satisfying the equations 
\begin{subequations}
\label{DNS_FEM_eqs}
\begin{align}
\label{DNS_FEM_eq1} 
&(D_{\tau} \rho_{h}^{n}, \varphi_{h}) + ( (P_h^{\rm RT} u_{h}^{n-1} 
\cdot \nabla) \rho_{h}^{n}, \varphi_{h} ) 
- \sum_{K \in \CT_{h}} \langle  P_h^{\rm RT} u_{h}^{n-1}  \cdot 
\lb \rho_{h}^{n} \rb, \varphi_{h} \rangle_{\partial K_{-}^{n}} = 0, \\[5pt]
\label{DNS_FEM_eq2} 
&(\chi (\rho_{h}^{n-1}) D_{\tau}u_{h}^{n},  v_{h}) 
+ \frac{1}{2} (  D_{\tau} \chi ( \rho_{h}^{n}) \,u_{h}^{n},  v_{h}) 
- \frac{1}{2}( \chi (\rho_{h}^{n})u_{h}^{n-1}, \nabla (u_{h}^{n} \cdot  v_{h})) \\ 
\nonumber 
&\qquad\qquad\qquad + (\chi (\rho_{h}^{n} )  (u_{h}^{n-1}\cdot \nabla) u_{h}^{n},  v_{h}) 
+ (\mu \nabla u_{h}^{n}, \nabla  v_{h}) - (p_{h}^{n}, \nabla\cdot  v_{h}) = 0,\\[5pt]
\label{DNS_FEM_eq3}
&(\nabla\cdot u_{h}^{n}, q_{h}) = 0,
\end{align}
\end{subequations}
for all test functions $(\varphi_{h},  v_{h}, q_{h}) \in {\rm P}^2_{\rm dG}(\CT_{h}) \times \mathring{\rm P}^{\rm 1b}(\CT_{h})^{d} \times \widetilde{\rm P}^1(\CT_{h})$, 
where 
$$
D_{\tau}\rho_{h}^{n} = \frac{\rho_{h}^{n} - \rho_{h}^{n-1}}{\tau},\quad
D_{\tau} \chi ( \rho_{h}^{n}):= \frac{\chi (\rho_{h}^{n}) - \chi (\rho_{h}^{n-1})}{\tau}
\quad\mbox{and}\quad
D_{\tau} u_{h}^{n} = \frac{u_{h}^{n} - u_{h}^{n-1}}{\tau} 
$$ 
are the backward Euler difference quotients of corresponding functions. 
The initial values of the numerical solutions are simply chosen to be 
$$
\rho_h^0 = P_h^{\rm dG}\rho^0
\quad\mbox{and}\quad
u_h^0 = I_hu^0 , 
$$
where $I_{h}: \mathring 
C(\overline\Omega)^d\rightarrow \mathring {\rm P}^{1b}(\CT_{h})^d$ is the globally continuous nodal interpolation operator.

The proposed method \eqref{DNS_FEM_eqs} has unconditional energy stability, i.e., substituting $\varphi_h=\rho_h^n$ and $v_h=u_h^n$ into \eqref{DNS_FEM_eqs}, and using the relation
\begin{align}\label{rho-test-identity}
& \sum_{K\in\CT_h} (P_h^{\rm RT} u_{h}^{n-1}  , \nabla \frac12 |\rho_h^n|^2)_{K}
- \sum_{K \in \CT_{h}} \langle  P_h^{\rm RT} u_{h}^{n-1}  \cdot 
\lb \rho_{h}^{n} \rb, \rho_h^n \rangle_{\partial K_{-}^{n}} \notag \\
&= \sum_{K\in\CT_h} \langle P_h^{\rm RT} u_{h}^{n-1}\cdot\nu_K  , \frac12 |\rho_h^n|^2  \rangle_{\partial K}
- \sum_{K \in \CT_{h}} \langle  P_h^{\rm RT} u_{h}^{n-1}  \cdot 
\lb \rho_{h}^{n} \rb, \rho_h^n \rangle_{\partial K_{-}^{n}} \notag \\
&= \sum_{K\in\CT_h} \langle P_h^{\rm RT} u_{h}^{n-1}\cdot\nu_K  , \frac12 |\rho_h^n|^2 \rangle_{\partial K_+}
+ \sum_{K\in\CT_h} \langle  P_h^{\rm RT} u_{h}^{n-1}\cdot\nu_K  , \frac12 |\rho_h^n|^2  \rangle_{\partial K_-} \notag \\
&\quad\, 
- \sum_{K \in \CT_{h}} \langle  P_h^{\rm RT} u_{h}^{n-1}  \cdot \nu_K 
[(\rho_{h}^{n})_--(\rho_{h}^{n})_+], (\rho_{h}^{n})_- \rangle_{\partial K_{-}^{n}} \notag \\
&=
- \sum_{K \in \CT_{h}} \langle  P_h^{\rm RT} u_{h}^{n-1}  \cdot \nu_K , \, 
\frac12 [(\rho_{h}^{n})_--(\rho_{h}^{n})_+]^2 \rangle_{\partial K_{-}^{n}} \notag \\
&\ge 0 ,
\end{align}
one can obtain the following energy inequality:
\begin{align}\label{energy-inequality}
&\frac12 \|\rho_h^n\|_{L^2(\Omega)}^2  
+ \int_\Omega \frac12 \chi (\rho_{h}^{n}) |u_{h}^{n}|^2 \d x
+ \tau \mu\|\nabla u_h^n\|_{L^2(\Omega)}^2 \notag \\ 
&\le 
\frac12 \|\rho_h^{n-1}\|_{L^2(\Omega)}^2
+ \int_\Omega \frac12 \chi (\rho_{h}^{n-1}) |u_{h}^{n-1}|^2 \d x .
\end{align}
Since \eqref{DNS_FEM_eqs} is a linearly implicit method, the energy inequality above implies existence and uniqueness of numerical solutions without any condition on the time stepsize or spatial mesh size (setting $\rho_h^{n-1}=0$ and $u_h^{n-1}=0$ in \eqref{energy-inequality} yields that the homogeneous linear system associated to \eqref{DNS_FEM_eqs} has only zero solution). 

In this article, we prove convergence of the numerical method \eqref{DNS_FEM_eqs} under the following regularity assumption on the exact solution: for some $\alpha\in(0,\frac12)$
\begin{equation}
\label{sol_reg}
\begin{aligned}
& \rho \in C([0,T]; H^{2+\alpha}(\Omega)),  
&&\partial_{t} \rho \in C([0,T]; H^{1}(\Omega)),  &&\partial_{t}^{2} \rho \in C([0,T]; L^{2}(\Omega)), \\
& \Vu \in C([0,T]; H^{2}(\Omega)),
&&\partial_{t} \Vu \in C([0,T]; H^{2}(\Omega)), 
&&\partial_{t}^{2} \Vu \in C([0,T]; L^{2}(\Omega)) ,\\
& p \in C([0,T]; H^{1}(\Omega)) ,  && \partial_t p \in C([0,T]; H^{1}(\Omega)) . 
\end{aligned}
\end{equation}
The spatial regularity in \eqref{sol_reg} is only slightly more than $H^2$, which is weaker and more reasonable than the regularity assumptions in \cite{Cai-Li-Li-2020} (which requires $H^3$ regularity of the solution) for this problem in a convex polygon or polyhedron. 
For the simplicity of notation, we denote by 
$$
u^n=u(\cdot,t_n) 
\quad\mbox{and}\quad 
\rho^n=\rho(\cdot,t_n)
$$ 
the exact solutions $u$ and $\rho$ at time level $t=t_n$. 

The main theoretical result of this article is the following theorem.

\begin{theorem}
\label{thm_proj_error} 
Under the regularity assumption {\rm\eqref{sol_reg}} and stepsize restriction $\tau = o(h^{d/2})$, there exists a positive constant $h_{*}$ such that when $h \le h_{*}$ the fully discrete solutions given by {\rm(\ref{DNS_FEM_eqs})} satisfy the following error estimate: 
\begin{align*}
 \max_{1\le n\le N} \left( \Vert  u^{n} - u_{h}^{n} \Vert_{L^{2}(\Omega)} 
+ \Vert \rho^{n} - \rho_{h}^{n} \Vert_{L^{2}(\Omega)} \right) 
+ \bigg(\sum_{n=1}^{N} \tau \Vert  u^{n} - u_{h}^{n} \Vert_{H^{1}(\Omega)}^{2}\bigg)^{\frac12}
\le C \big( \tau + h^{\frac{3}{2} + \alpha} \big) , 
\end{align*} 
where constant $C$ may depend on the exact solution $(\rho, u, p)$ and $T$. 
\end{theorem}

The proof of Theorem \ref{thm_proj_error} is presented in the next section. 
Throughout, we denote by $C$ a generic positive constant that may be different at different occurrences and may depend on the exact solution $(\rho, u, p)$ and $T$, but is independent of the mesh size $h$ and stepsize $\tau$. 

\begin{remark}
The convergence rates in Theorem \ref{thm_proj_error} is limited by the regularity of solutions and the nature of hyperbolic equation of $\rho$. It is known that even for linear hyperbolic equations, the DG method generally loses half-order convergence; see \cite[Corollary 2.32]{DiPietroErn2012}. 
Once the error estimates for velocity and density are obtained, a weaker error estimate for the pressure (losing additional half order in time and one order in space) can be obtained by using the method in \cite{Cai-Li-Li-2020}, which we omit in this paper. An error estimate for pressure without losing additional order of accuracy is still missing for this problem even in two dimensions.  
\end{remark}

\section{Error analysis}\label{sec:proof}

\subsection{Preliminary results}

We denote by $\CF_{h}^{I}$ and $\CF_{h}^{\partial}$ the set of all interior and boundary faces of $\CT_{h}$, respectively, and define $\CF_{h} := \CF_{h}^{I} \cup \CF_{h}^{\partial}$ to be the collection of all faces. For an interior face $F = \partial K\cap \partial K^{\prime}$ with $K,K^{\prime}\in\CT_h$, the average and jump defined in \eqref{average-jump}, initially defined on $\partial K$ and $\partial K'$, respectively, coincide on the face $F$ and can be rewritten as   
\begin{align*}
\Lb \phi \Rb := \frac{1}{2}(\phi + \phi^{\prime})
\quad\mbox{and}\quad
\lb \phi \rb  := \phi\nu_{K} + \phi^{\prime} \Vn_{K^{\prime}} \quad\mbox{on}\quad F, 
\end{align*}
where $\phi$ and $\phi^{\prime}$ denote the trace of $\phi$ from the interior of $K$ and $K^{\prime}$, respectively. If $F \in \CF_{h}^{\partial}$, then we define the average and jump of $\phi$ on $F$ as
\begin{align*}
\Lb \phi \Rb := \phi\quad\mbox{and}\quad
\lb \phi \rb := \phi \, \Vn .
\end{align*}

We denote by $\Pi_{h}^{\rm RT}$ the standard Raviart--Thomas projection from $H^{1}(\Omega)^{d}$ 
onto ${\rm RT}^1(\CT_{h})$, which has the following properties (cf. \cite[Lemma 17.1]{Thomee2006}) 
\begin{align}
&(\nabla\cdot \Pi_{h}^{\rm RT}v ,w_h) = (\nabla\cdot v ,w_h)
\quad\forall\, w_h\in P^1_{\rm dG}(\CT_h),
&&\forall\, v\in H(\text{div},\Omega) \cap H^{1}(\CT_{h})^{d} , \label{RT-proj-1} \\
& 
\Pi_{h}^{\rm RT}v\cdot\nu\in {\rm P}^1_{\rm dG}(\partial K),
\quad
\int_{\partial K}
\Pi_{h}^{\rm RT}v\cdot\nu \,w_h\d s
=\int_{\partial K}
v\cdot\nu \,w_h\d s &&\forall\, w_h\in {\rm P}^1_{\rm dG}(\partial K), \label{RT-proj-2} \\
&\Vert v - \Pi_{h}^{\rm RT} v \Vert_{L^{2}(\Omega)} 
\le C h^{l} \Vert v \Vert_{H^{l}(\Omega)} 
&&\forall\, v\in H^{l}(\Omega)^d ,\,\, l=1,2, \label{RT-proj-3}
\end{align} 
where ${\rm P}^1_{\rm dG}(\partial K)$ denotes the space of piecewise linear functions on $\partial K$ (possibly discontinuous at the vertices). 

We recall that 
$$
u^n=u(\cdot,t_n) 
\quad\mbox{and}\quad 
p^n=p(\cdot,t_n)
$$ 
the exact solutions $u$ and $p$ at time level $t=t_n$.
Since $\nabla\cdot u^{n} = 0$ in $\Omega$ and $u^{n}\cdot \nu = 0$ on $\partial\Omega$, 
it follows from \eqref{RT-proj-1}--\eqref{RT-proj-2} that 
$$
\nabla\cdot \Pi_{h}^{\rm RT} u^n
=0 
\,\,\,\mbox{in}\,\,\,\Omega 
\quad\mbox{and}\quad
\Pi_{h}^{\rm RT} u^n \cdot \nu =0\,\,\,\mbox{on}\,\,\,\partial\Omega.
$$
This implies $\Pi_{h}^{\rm RT} u^n\in {\rm RT}^1_0(\CT_{h})$ in view of the definition in \eqref{FEM_space_post_processed_velocity}. Since the $L^2$ projection $P_h^{\rm RT} u^{n}$ is the element in ${\rm RT}^1_0(\CT_{h})$ closest to $u^{n}$ in the $L^2$ norm, it follows that 
\begin{align}
\label{postprocessing_appr}
\Vert P_h^{\rm RT} u^{n} - u^{n}\Vert_{L^{2}(\Omega)} \le 
\Vert \Pi_{h}^{\rm RT}u^{n} - u^{n}\Vert_{L^{2}(\Omega)} 
\le C h^{l} \Vert u^{n} \Vert_{H^{l}(\Omega)} ,
\quad l=1,2, 
\end{align}
where the last inequality is due to \eqref{RT-proj-3}. This estimate of $\Vert P_h^{\rm RT} u^{n} - u^{n}\Vert_{L^{2}(\Omega)}$ will be used in the error analysis. 

Let $(\widehat u_h^n,\widehat p_h^n)\in \mathring {\rm P}^{\rm 1b}(\CT_{h})^d \times \widetilde {\rm P}^{\rm 1}(\CT_{h})$ be Stokes--Ritz projection of the exact solution $(u^n,p^n)\in H^1_0(\Omega)^d\times \widetilde L^2(\Omega)$, defined by 
\begin{align}\label{Stokes-Ritz-Proj}
\left\{
\begin{aligned}
&(\nabla \widehat u_h^n,\nabla v_h) - (\widehat p_h^n,\nabla\cdot v_h) 
=(\nabla u^n,\nabla v_h) - (p^n,\nabla\cdot v_h) 
&&\forall\, v_h\in \mathring {\rm P}^{\rm 1b}(\CT_{h})^d , \\
&(\nabla\cdot \widehat u_h^n,q_h)
=(\nabla\cdot u^n,q_h) 
&&\forall\, q_h\in \widetilde {\rm P}^{\rm 1}(\CT_{h}).
\end{aligned}
\right.
\end{align}
It is known that the Stokes--Ritz projection has the following approximation property (cf. \cite{Arnold-Brezzi-Fortin-1984,Boffi-2008}):
\begin{align}\label{assumption_stokes_proj} 
\Vert u^n-\widehat u_h^n\Vert_{L^{2}(\Omega)} 
+\Vert p^n-\widehat p_h^n\Vert_{L^{2}(\Omega)} 
\le
C h^{2} \left( \Vert u^n\Vert_{H^{2}(\Omega)} 
+ \Vert p^n\Vert_{H^{1}(\Omega)}  \right) .
\end{align}
Note that all finite element functions satisfy the following ``inverse inequality" (see \cite[\textsection 4.5]{Brenner-Scott}): 
$$
\| w_h \|_{W^{s_1,q}(\Omega)} \le Ch^{s_2-s_1+\frac{d}{q}-\frac{d}{p}} \| w_h \|_{W^{s_2,p}(\Omega)} 
\quad\mbox{for}\,\,\, 0\le s_2\le s_1\le 1,\,\,\, 1\le p\le q\le\infty . 
$$
where the constant $C$ depending on the finite element space of $w_h$ (but independent of $h$).
In addition, the Lagrange interpolation $I_{h}: \mathring 
C(\overline\Omega)^d\rightarrow \mathring {\rm P}^{1b}(\CT_{h})^d$ has the following error bound (see \cite[\textsection 4.4, Corollary 4.4.7]{Brenner-Scott}):   
$$
\| u^n -I_h u^n \|_{L^2(\Omega)} \le Ch^{2} \| u^n \|_{H^{2}(\Omega)}
\quad\mbox{anbd}\quad 
\| u^n -I_h u^n \|_{L^\infty(\Omega)} \le Ch^{2-\frac{d}{2}} \| u^n \|_{H^{2}(\Omega)} . 
$$
By using the two estimates above, from \eqref{assumption_stokes_proj} one can obtain 
\begin{align}\label{assumption_stokes_proj_Linfty} 
\Vert u^n-\widehat u_h^n\Vert_{L^{\infty}(\Omega)} 
& 
\le \Vert u^n-I_h u^n\Vert_{L^{\infty}(\Omega)} 
+\Vert I_h u^n-\widehat u_h^n\Vert_{L^{\infty}(\Omega)} 
\notag\\ 
&
\le Ch^{2-\frac{d}{2}} \| u^n \|_{H^{2}(\Omega)} 
+ Ch^{-\frac{d}{2}} \Vert I_h u^n-\widehat u_h^n\Vert_{L^{2}(\Omega)} 
 \notag\\ 
&
\le Ch^{2-\frac{d}{2}} \| u^n \|_{H^{2}(\Omega)} 
+ Ch^{-\frac{d}{2}} \Vert I_h u^n-u^n\Vert_{L^{2}(\Omega)} 
+ Ch^{-\frac{d}{2}} \Vert u^n-\widehat u_h^n\Vert_{L^{2}(\Omega)} 
\notag\\ 
&\le
C h^{2-\frac{d}{2}} 
\le
Ch^{\frac12},\quad\mbox{for}\,\,\, d=2,3. 
\end{align}
The Stokes--Ritz projection $(\widehat u_h^n,\widehat p_h^n)$ will serve as an intermediate solution for comparison with the numerical solution $(u_h^n,p_h^n)$. With the approximation property \eqref{assumption_stokes_proj}, it suffices to estimate the error $e_u^n=u_h^n-\widehat u_h^n$ and $e_p^n=p_h^n-\widehat p_h^n$ for the velocity equation. 

To control the coupling term in the hyperbolic density equation, the following discrete Sobolev embedding inequality will be used in the error analysis. 

\begin{lemma}
\label{lemma_postprocessing_lp}
In a convex polyhedron (or polygon) $\Omega$, the following inequality holds: 
\begin{align}
\label{postprocessing_lp}
\Vert P_h^{\rm RT} v \Vert_{L^{6}(\Omega)} \le C \Vert v\Vert_{H^{1}(\Omega)}
\quad\forall\, v \in H^{1}(\Omega)^{d} 
\,\,\,\mbox{and\,\,\,$v\cdot\nu=0$ on $\partial\Omega$.} 
\end{align}
\end{lemma}

\begin{proof}
Let ${\rm RT}^1_\nu(\CT_{h})$ be the subspace of ${\rm RT}^1(\CT_{h})$ with the boundary condition $\sigma_h\cdot\nu=0$ on $\partial\Omega$ for $\sigma_h\in {\rm RT}^1_\nu(\CT_{h})$. We define $(\sigma_{h}, \phi_{h}) \in 
{\rm RT}^1_\nu(\CT_{h}) \times {\rm P}^1_{\rm dG}(\CT_{h})$ to be the solution of the following mixed finite element equations: 
\begin{subequations}
\label{postprocessing_mixed_method_eqs}
\begin{align}
\label{postprocessing_mixed_method_eq1}
& (\sigma_{h}, \eta_{h}) 
+ (\phi_{h}, \nabla\cdot \eta_{h}) 
= (v, \eta_{h}) 
&&\hspace{-50pt} \forall\,\eta_h\in {\rm RT}^1_\nu(\CT_{h}) , \\ 
\label{postprocessing_mixed_method_eq2}
& (\nabla \cdot \sigma_{h}, \varphi_{h}) = 0
&&\hspace{-50pt} \forall\,\varphi_h\in {\rm P}^1_{\rm dG}(\CT_{h}) .
\end{align}
\end{subequations}
The second equation above implies $\nabla \cdot \sigma_{h}=0$. This together with the boundary condition $\sigma_h\cdot\nu=0$ on $\partial\Omega$ implies that $\sigma_h\in {\rm RT}^1_0(\CT_{h})$, which is defined in \eqref{FEM_space_post_processed_velocity}. By choosing $\eta_h\in {\rm RT}^1_0(\CT_{h})$ in the first equation we obtain $\sigma_h = P_h^{\rm RT} v$, i.e., the $L^2$ projection of $v$ onto ${\rm RT}^1_0(\CT_{h})$. 

The partial differential equations to which the mixed method 
(\ref{postprocessing_mixed_method_eqs}) approximates is 
\begin{align*}
&\sigma - \nabla \phi = v, \\
&\nabla\cdot \sigma = 0,
\end{align*}
with boundary condition $\sigma\cdot\nu = 0$ on $\partial\Omega$. Thus 
\begin{align*}
\left\{
\begin{aligned}
- \Delta \phi &= \nabla\cdot v && \text{ in } \Omega , \\
-\partial_\nu \phi &= 0 && \text{ on } \partial \Omega.
\end{aligned}
\right.
\end{align*}
By the regularity of the Neumann problem in a convex polyhedron (cf. \cite[Theorem 3.2.1.3 and Theorem 3.1.3.3]{Grisvard2011}), we have 
\begin{align*}
\Vert \phi \Vert_{H^2(\Omega)} 
\le C\Vert \nabla\cdot v\Vert_{L^2(\Omega)}
\le C\Vert v\Vert_{H^{1}(\Omega)}. 
\end{align*}
By the standard error estimate of the mixed FEM (cf. \cite[Theorem 17.1]{Thomee2006}), we have 
\begin{align*}
\Vert \sigma_{h} - \Pi_h^{\rm RT}\sigma \Vert_{L^{2}(\Omega)} 
\le C h \Vert \phi\Vert_{H^2(\Omega)} 
\le C h \Vert v\Vert_{H^{1}(\Omega)}. 
\end{align*}
This implies that, via the inverse inequality,  
\begin{align}\label{sigmah-Phsigma}
\Vert \sigma_{h} - \Pi_h^{\rm RT}\sigma \Vert_{L^{6}(\Omega)}
\le Ch^{-\frac{d}{3}}\Vert \sigma_{h} - \Pi_h^{\rm RT}\sigma \Vert_{L^{2}(\Omega)} 
\le C h^{1-\frac{d}{3}} \Vert v\Vert_{H^{1}(\Omega)}. 
\end{align}
Then, using the $L^{2}$-orthogonal projection $P_h^{\rm dG}:L^2(\Omega)^d\rightarrow {\rm P}^2_{\rm dG}(\CT_{h})^{d}$. 
In the case $d\in\{2,3\}$ we obtain, by using the triangle inequality, 
\begin{align*}
\Vert \sigma_{h} \Vert_{L^{6}(\Omega)}
&\le
\Vert \sigma_{h} - \Pi_h^{\rm RT}\sigma \Vert_{L^{6}(\Omega)}
+\Vert \Pi_h^{\rm RT}\sigma - P_h^{\rm dG} \sigma \Vert_{L^{6}(\Omega)} 
+ \Vert P_h^{\rm dG}  \sigma \Vert_{L^{6}(\Omega)} \\
&\le
\Vert \sigma_{h} - \Pi_h^{\rm RT}\sigma \Vert_{L^{6}(\Omega)} 
+ C h^{-\frac{d}{3}} \Vert \Pi_h^{\rm RT}\sigma  - P_h^{\rm dG} \sigma \Vert_{L^{2}(\Omega)}
+C \Vert \sigma \Vert_{H^{1}(\Omega)} 
\le C \Vert v\Vert_{H^{1}(\Omega)} , 
\end{align*}
where the last inequality uses \eqref{sigmah-Phsigma} and \eqref{postprocessing_appr}. This proves the desired result in 
Lemma \ref{lemma_postprocessing_lp}. 
\end{proof}

Let $H^{1}(\CT_{h})$ be the broken $H^1$ space, consisting of functions which are in $H^1(K)$ for all tetrahedra $K\in\CT_h$, 
equipped with the norm 
\begin{align}
\label{H1_discrete} 
\Vert \varphi \Vert_{H^{1}(\CT_{h})} := 
\bigg( \sum_{K\in\CT_h} \Vert \nabla \varphi \Vert_{L^{2}(K)}^{2} 
+ \sum_{F \in \mathcal{F}_{h}} h_{F}^{-1} \Vert \lb \varphi \rb \Vert_{L^{2}(F)}^{2} \bigg)^{\frac{1}{2}} , 
\end{align}
where $h_F$ denotes the diameter of face $F$, equivalent to the diameter of tetrahedron $K$ containing face $F$ according to the shape regularity of the partition. 
The $H^{1}(\CT_{h})$-stability of the $L^2$-orthogonal projection $P_h^{\rm dG}$ is presented in the following lemma. 

\begin{lemma}
\label{lemma_P_rho_H1}
The $L^2$ projection operator $P_h^{\rm dG}:L^2(\Omega)\rightarrow {\rm P}^2_{\rm dG}(\CT_{h})$ defined in \eqref{L2-Proj-dG} satisfies the following estimate: 
\begin{align*}
\Vert P_h^{\rm dG} \varphi \Vert_{H^{1}(\CT_{h})} 
\le C \Vert \varphi \Vert_{H^{1}(\CT_{h})} \quad 
\forall\, \varphi \in H^{1}(\CT_{h}).  
\end{align*}
\end{lemma}

\begin{proof}
For any $K \in \CT_{h}$ the following standard $L^2$ and $H^1$ approximation properties hold: 
\begin{align*}
\Vert P_h^{\rm dG} \varphi - \varphi\Vert_{L^{2}(K)} \le C h_{K}\Vert \nabla \varphi \Vert_{L^{2}(K)} 
\text{  and  } \Vert \nabla (P_h^{\rm dG} \varphi - \varphi) \Vert_{L^{2}(K)}  
\le C \Vert \nabla \varphi \Vert_{L^{2}(K)}. 
\end{align*}
By the trace inequality on the tetrahedron $K$ and the above approximation properties, we have 
\begin{align*}
h_{F}^{-1} \Vert P_h^{\rm dG} \varphi - \varphi \Vert_{L^{2}(\partial K)}^{2} 
\le  C \left( h_{K}^{-2} \Vert P_h^{\rm dG} \varphi - \varphi\Vert_{L^{2}(K)}^{2}  
+ \Vert \nabla (P_h^{\rm dG} \varphi - \varphi) \Vert_{L^{2}(K)}^{2}\right) 
\le  C \Vert \nabla \varphi \Vert_{L^{2}(K)}^{2}. 
\end{align*}
Hence, 
\begin{align*}
\Vert P_h^{\rm dG} \varphi - \varphi \Vert_{H^{1}(\CT_{h})}^{2} 
& =
\sum_{K\in \CT_{h}} 
\big( \Vert \nabla (P_h^{\rm dG} \varphi - \varphi ) \Vert_{L^{2}(K)}^{2} 
+ h_{F}^{-1} \Vert \lb P_h^{\rm dG} \varphi - \varphi  \rb 
\Vert_{L^{2}(\partial K)}^{2} \big)  \\
&\le 
C \sum_{K \in \CT_{h}} \Vert \nabla \varphi \Vert_{L^{2}(K)}^{2}
\le
C \Vert \varphi \Vert_{H^{1}(\CT_{h})}^{2}  . 
\end{align*}
The desired result follows from the above inequality and the triangle inequality. 
\end{proof}

\subsection{Mathematical induction}

We define the following error functions: 
\begin{align*}
e_{\rho, h}^{n} = P_h^{\rm dG} \rho^{n} - \rho_{h}^{n}, \quad 
 e_{u,h}^{n} = \widehat u_{h}^{n} - u_{h}^{n}, \quad 
e_{p,h}^{n} = \widehat p_{h}^{n}  - p_{h}^{n}.
\end{align*} 

For a given $1 \le m \le N$, we assume that the data $\rho_{h}^{n-1}$ and $u_{h}^{n-1}$, $n=1,2,\cdots, m$ are given and satisfying the following inequalities (errors on the previous time level are sufficiently small in some sense): 
\begin{subequations}
\label{pre_step_errors}
\begin{align}
\label{pre_step_error1}
\max_{1\le n \le m} \Vert e_{\rho,h}^{n-1}\Vert_{L^{\infty}(\Omega)} 
&\le 
\frac{1}{4} \rho_{\rm min}, \\
\label{pre_step_error2}
\max_{1\le n \le m} \Vert  e_{u,h}^{n-1}\Vert_{L^{2}(\Omega)} 
&\le  
h^{\frac{3}{2}+\frac{\alpha}{2}}  + \tau^{\frac56}  , \\ 
\label{pre_step_error3}
\max_{1\le n \le m} \Vert  e_{u,h}^{n-1}\Vert_{L^{\infty}(\Omega)} 
&\le 1 \\
\label{pre_step_error5}
\max_{1\le n \le m} \Vert  P_h^{\rm RT} u_{h}^{n-1} - u^{n-1}\Vert_{L^{\infty}(\Omega)}
&\le 2 , \\
\label{pre_step_error4}
\sum_{n=1}^{m} \tau \Vert  e_{\Vu, h}^{n-1} \Vert_{H^{1}(\Omega)}^{2} 
&\le (\kappa+h^{\alpha}) h^3  , 
\end{align}
\end{subequations}
where $\kappa$ is a sufficiently small constant to be determined later in \eqref{rho_errors}--\eqref{proved_now_step_error1}. Then we prove that the numerical solution $(\rho_{h}^{m}, u_{h}^{m}, p_{h}^{m}) \in {\rm P}^2_{\rm dG}(\CT_{h}) \times \mathring{\rm P}^{\rm 1b}(\CT_{h}) \times \widetilde{\rm P}^1(\CT_{h})$ given by (\ref{DNS_FEM_eqs}) satisfies the following inequalities: 
\begin{subequations}
\label{now_step_errors}
\begin{align}
\label{now_step_error1}
\max_{0 \le n \le m} \Vert e_{\rho,h}^{n}\Vert_{L^{\infty}(\Omega)} 
&\le \frac{1}{4} \rho_{\rm min}, \\
\label{now_step_error2}
\max_{0 \le n \le m} \Vert  e_{u,h}^{n}\Vert_{L^{2}(\Omega)} 
&\le 
h^{\frac{3}{2}+\frac{\alpha}{2}}  + \tau^{\frac56}  , \\ 
\label{now_step_error3}
\max_{0 \le n \le m} \Vert  e_{u,h}^{n}\Vert_{L^{\infty}(\Omega)} 
&\le 1,\\
\label{now_step_error5}
\max_{0\le n \le m} \Vert  P_h^{\rm RT} u_{h}^{n} - u^{n}\Vert_{L^{\infty}(\Omega)}
&\le 2 ,\\
\label{now_step_error4}
\sum_{n=0}^{m} \tau \Vert  e_{u,h}^{n} \Vert_{H^{1}(\Omega)}^{2} 
&\le (\kappa+h^{\alpha}) h^3   .
\end{align}
\end{subequations} 
If this can be proved then, by mathematical induction, \eqref{now_step_errors} holds for all $1\le n\le N$. 
To use mathematical induction, we emphasize that all the generic constants below will be independent of $m$ 
(but may depend on $T$).

The induction assumption (\ref{pre_step_error1}) implies that 
\begin{align*}
& \Vert I_{h}^{\rm dG} \rho^{n-1} - \rho_{h}^{n-1} \Vert_{L^{\infty}(\Omega)} \\
& \le \Vert I_{h}^{\rm dG} \rho^{n-1} - P_h^{\rm dG}\rho^{n-1} \Vert_{L^{\infty}(\Omega)} 
+ \Vert e_{\rho, h}^{n-1}\Vert_{L^{\infty}(\Omega)} \\
& \le C h^{-\frac{d}{2}} \Vert I_{h}^{\rm dG} \rho^{n-1} - P_h^{\rm dG}
\rho^{n-1} \Vert_{L^{2}(\Omega)} + \Vert e_{\rho, h}^{n-1}\Vert_{L^{\infty}(\Omega)} \\
& \le C h^{2-\frac{d}{2}} \Vert \rho^{n-1} \Vert_{H^{2}(\Omega)} + 
\frac{1}{4} \rho_{\rm min} 
\le \frac{3}{8} \rho_{\rm min}, 
\quad\text{when $h$ is sufficiently small.}
\end{align*}
Since the nodal interpolation $I_{h}^{\rm dG}\rho^{n-1}$ satisfies 
\begin{align*}
\|\rho^{n-1}-I_{h}^{\rm dG}\rho^{n-1}\|_{L^\infty(\Omega)}
&\le Ch^{2-\frac{d}{2}}\|\rho^{n-1}\|_{C^{\frac12}(\Omega)} \\
&\le Ch^{2-\frac{d}{2}}\|\rho^{n-1}\|_{H^{2+\alpha}(\Omega)} \quad\mbox{(Sobolev embedding $H^{2+\alpha}(\Omega)\hookrightarrow C^{\frac12}(\Omega)$)} \\
&\le \frac{1}{8} \rho_{\rm min} \quad\mbox{when $h$ is sufficiently small},
\end{align*}
it follows that (by using the triangle inequality)
\begin{align*}
\|\rho^{n-1}-\rho_h^{n-1}\|_{L^\infty(\Omega)}
&\le \frac{1}{2}\rho_{\rm min} ,
\end{align*}
which implies 
\begin{align}\label{rho_inf_sup_pre}
\frac{1}{2}\rho_{\rm min} \le \rho_h^{n-1}(x) 
\le \frac{3}{2} \rho_{\rm max} ,
\quad n = 1,\cdots, m , 
\end{align}
in view of the definition of $\rho_{\min}$ and $\rho_{\max}$ in \eqref{def-rho-min-max}. 

Similarly, the error estimate \eqref{assumption_stokes_proj} for the Stokes--Ritz projection and \eqref{pre_step_error3} imply that 
\begin{align}\label{u_inf_sup_pre}
\Vert u_{h}^{n-1}\Vert_{L^{\infty}(\Omega)}
&\le  \Vert e_{u,h}^{n-1} \Vert_{L^{\infty}(\Omega)}
+ \Vert \widehat u_h^{n-1} - I_{h}u^{n-1} 
\Vert_{L^{\infty}(\Omega)} + \Vert I_{h}u^{n-1} - u^{n-1} \Vert_{L^{\infty}(\Omega)} 
+ \Vert u^{n-1} \Vert_{L^{\infty}(\Omega)} \nonumber \\ 
\nonumber
&\le  1
+ h^{-\frac{d}{2}} \Vert \widehat u_h^{n-1} - I_{h}u^{n-1} 
\Vert_{L^2(\Omega)}
+ C h^{\frac{1}{4}} \Vert u^{n-1} \Vert_{C^{\frac14}(\Omega)} + \Vert u^{n-1} \Vert_{L^{\infty}(\Omega)}\\ 
\nonumber
&\le 1
+ C h^{2-\frac{d}{2}} \left(\Vert \Vu^{n-1} \Vert_{H^{2}(\Omega)} 
+ \Vert p^{n-1}\Vert_{H^{1}(\Omega)} \right)
 + C h^{\frac{1}{4}} \Vert u^{n-1} \Vert_{H^{2}(\Omega)} + \Vert u^{n-1} \Vert_{L^{\infty}(\Omega)}\\
 &\le 2 + \Vert u^{n-1} \Vert_{L^{\infty}(\Omega)}
 \quad\mbox{(when $h$ is sufficiently small)} .
\end{align}
Meanwhile, \eqref{pre_step_error5} implies 
\begin{align}
\label{postprocessing_u_inf_sup_pre}
& 
\max_{1\le n\le m}
\Vert  P_h^{\rm RT} u_{h}^{n-1}\Vert_{L^{\infty}(\Omega)} 
\le 2 + \Vert u\Vert_{L^\infty(0,T;L^{\infty}(\Omega))}  .
\end{align}
The boundedness of numerical solutions in \eqref{rho_inf_sup_pre}--\eqref{postprocessing_u_inf_sup_pre} will be used in the following error analysis in estimating the nonlinear terms. 

\subsection{Estimates for $e_{\rho, h}^{n}$}

From (\ref{DNS_eq1}) we know that the exact solution $\rho^n$ satisfies the equation 
\begin{align}
\label{DNS_time_discrete_eq1}
(D_{\tau}\rho^{n}, \varphi_{h}) + (\Vu^{n-1}\cdot \nabla \rho^{n}, \varphi_{h}) 
= (R^n_\rho , \varphi_{h})
\quad \forall\, \varphi_{h} \in {\rm P}^2_{\rm dG}(\CT_{h}) 
\end{align}
with 
\begin{align*}
R^n_\rho =  D_{\tau}\rho^{n} - \partial_{t} \rho^{n}  
+  (\Vu^{n-1} -  u^{n})\cdot \nabla \rho^{n} .
\end{align*}
Subtracting (\ref{DNS_FEM_eq1}) from (\ref{DNS_time_discrete_eq1}) yields 
\begin{align}
\label{DNS_error_eq1_old}
& (D_{\tau}(\rho^{n} - P_h^{\rm dG}\rho^{n}), \varphi_{h}) 
+ (D_{\tau}e_{\rho, h}^{n}, \varphi_{h}) 
\nonumber \\
\nonumber
& \quad + ((P_h^{\rm RT}u_{h}^{n-1})\cdot \nabla 
(\rho^{n} - P_h^{\rm dG}\rho^{n}), \varphi_{h}) 
+ ((P_h^{\rm RT}u_{h}^{n-1})\cdot \nabla e_{\rho, h}^{n}, \varphi_{h}) \\
\nonumber
&\quad 
- \sum_{K \in \CT_{h}} \langle  P_h^{\rm RT} u_{h}^{n-1}  \cdot 
\lb \rho^n-P_h^{\rm dG}\rho^n \rb, \varphi_{h} \rangle_{\partial K_{-}^{n}} 
 - \sum_{K \in \CT_{h}} \langle  P_h^{\rm RT} u_{h}^{n-1}  \cdot 
\lb e_{\rho,h}^{n} \rb, \varphi_{h} \rangle_{\partial K_{-}^{n}} 
\nonumber \\
& \quad + ((\Vu^{n-1} - P_h^{\rm RT}\Vu^{n-1})\cdot 
\nabla \rho^{n},\varphi_{h}) 
+ (P_h^{\rm RT}(\Vu^{n-1} - u_{h}^{n-1})\cdot 
\nabla \rho^{n}, \varphi_{h}) \nonumber \\
& =  (R^n_\rho , \varphi_{h})
\qquad \forall \varphi_{h} \in {\rm P}^2_{\rm dG}(\CT_{h}).
\end{align}

On a face $F\in \mathcal{F}_{h}^{I}$ we denote by $\widehat{P_h^{\rm dG}\rho^{n}}$ the value of 
$P_h^{\rm dG}\rho^{n}$ from the in-flow side. Then, by using integration by parts, we have
\begin{align*}
& (P_h^{\rm RT}u_{h}^{n-1} \cdot \nabla 
(\rho^{n} - P_h^{\rm dG}\rho^{n}), \varphi_{h}) - \sum_{K \in \CT_{h}} \langle  P_h^{\rm RT} u_{h}^{n-1}  \cdot 
\lb \rho^n-P_h^{\rm dG}\rho^n \rb, \varphi_{h} \rangle_{\partial K_{-}^{n}}   \\
&=   - ( P_h^{\rm RT}u_{h}^{n-1} 
 (\rho^{n} - P_h^{\rm dG}\rho^{n}), \nabla \varphi_{h}) 
+\sum_{K \in \CT_{h}} \langle ( P_h^{\rm RT} u_{h}^{n-1} \cdot \nu_{K} ) (\rho^{n} - \widehat{P_h^{\rm dG}\rho^{n}}), \varphi_{h} \rangle_{\partial K} \\
&\hspace{295pt}
\quad\forall\, \varphi_{h} \in {\rm P}^2_{\rm dG}(\CT_{h}) .
\end{align*}
Then, substituting this identity into (\ref{DNS_error_eq1_old}), we obtain  
\begin{align*}
& (D_{\tau}(\rho^{n} - P_h^{\rm dG}\rho^{n}), \varphi_{h}) 
+ (D_{\tau}e_{\rho, h}^{n}, \varphi_{h}) \nonumber \\[2pt] 
\nonumber 
&\quad\, - (P_h^{\rm RT}u_{h}^{n-1} 
 (\rho^{n} - P_h^{\rm dG}\rho^{n}), \nabla \varphi_{h}) \nonumber \\[2pt] 
&\quad\,
+\sum_{K \in \CT_{h}} \langle ( P_h^{\rm RT} u_{h}^{n-1} 
\cdot \nu_{K} ) (\rho^{n} - \widehat{P_h^{\rm dG}\rho^{n}}), \varphi_{h} \rangle_{\partial K} \nonumber \\ 
&\quad\,
+ (P_h^{\rm RT}u_{h}^{n-1}\cdot \nabla e_{\rho, h}^{n}, \varphi_{h})  - \sum_{K \in \CT_{h}} \langle  P_h^{\rm RT} u_{h}^{n-1}  \cdot 
\lb e_{\rho,h}^{n} \rb, \varphi_{h} \rangle_{\partial K_{-}^{n}} \nonumber \\
&\quad\, + ((\Vu^{n-1} - P_h^{\rm RT}\Vu^{n-1})\cdot 
\nabla \rho^{n},\varphi_{h}) 
+ (P_h^{\rm RT}(\Vu^{n-1} - u_{h}^{n-1})\cdot 
\nabla \rho^{n}, \varphi_{h}) \nonumber \\[5pt]
&=  ( R^n_{\rho} , \varphi_{h})
\qquad \forall \varphi_{h} \in {\rm P}^2_{\rm dG}(\CT_{h}) ,
\end{align*}
which can be rewritten as 
\begin{align}
\label{DNS_error_eq1}
& (D_{\tau}e_{\rho, h}^{n}, \varphi_{h}) 
+ (P_h^{\rm RT}u_{h}^{n-1}\cdot \nabla e_{\rho, h}^{n}, \varphi_{h})  - \sum_{K \in \CT_{h}} \langle  P_h^{\rm RT} u_{h}^{n-1}  \cdot 
\lb e_{\rho,h}^{n} \rb, \varphi_{h} \rangle_{\partial K_{-}^{n}} \nonumber \\
&=
-(D_{\tau}(\rho^{n} - P_h^{\rm dG}\rho^{n}), \varphi_{h}) 
+ (P_h^{\rm RT}u_{h}^{n-1} 
 (\rho^{n} - P_h^{\rm dG}\rho^{n}), \nabla \varphi_{h}) \nonumber \\[2pt] 
&\quad\,
-\sum_{K \in \CT_{h}} \langle ( P_h^{\rm RT} u_{h}^{n-1} 
\cdot \nu_{K} ) (\rho^{n} - \widehat{P_h^{\rm dG}\rho^{n}}), \varphi_{h} \rangle_{\partial K} \nonumber \\ 
&\quad\,
- ((\Vu^{n-1} - P_h^{\rm RT}\Vu^{n-1})\cdot 
\nabla \rho^{n},\varphi_{h}) 
- (P_h^{\rm RT}(\Vu^{n-1} - u_{h}^{n-1})\cdot 
\nabla \rho^{n}, \varphi_{h}) \nonumber \\
&\quad\,
+   ( R^n_{\rho} , \varphi_{h}) \nonumber \\
&=: \sum_{j=1}^6 E_j^n(\varphi_h) .
\end{align}
Since $\nabla\cdot (P_h^{\rm RT}u_{h}^{n-1}) = 0$  
and $(P_h^{\rm RT}u_{h}^{n-1}) \cdot \Vn |_{\partial \Omega} = 0$, 
it can be verified that
\begin{align}
\label{dg_energy}
 & (P_h^{\rm RT}u_{h}^{n-1}\cdot \nabla e_{\rho, h}^{n}, e_{\rho,h}^n)  - \sum_{K \in \CT_{h}} \langle  P_h^{\rm RT} u_{h}^{n-1}  \cdot 
\lb e_{\rho,h}^{n} \rb, e_{\rho,h}^n \rangle_{\partial K_{-}^{n}}  \\ 
\nonumber
& = \dfrac{1}{2} \sum_{F\in \mathcal{F}_{h}^{I}} \Vert 
\vert (P_h^{\rm RT}u_{h}^{n-1})\cdot \Vn_{F}\vert^{\frac{1}{2}} 
 \lb e_{\rho, h}^{n} \rb \Vert_{L^{2}(F)}^{2} , 
\end{align}
which is similar as \eqref{rho-test-identity}.  
Since $E_1^n(\varphi_h)=(D_{\tau}\rho^{n} 
- P_h^{\rm dG}(D_{\tau}\rho^{n}), \varphi_{h}) = 0$ for any $\varphi_{h}\in {\rm P}^2_{\rm dG}(\CT_{h})$, substituting  
$\varphi_{h} = e_{\rho, h}^{n}$ into (\ref{DNS_error_eq1}) yields
\begin{align}\label{Error-ineq-e-rho}
& \frac{1}{2} D_{\tau}\big( \Vert e_{\rho, h}^{n}\Vert_{L^{2}(\Omega)}^{2} \big) 
+ \frac{1}{2} \sum_{F\in \mathcal{F}_{h}^{I}} \big \Vert \vert (P_h^{\rm RT}u_{h}^{n-1})\cdot \Vn_{F}\vert^{\frac{1}{2}} 
\lb e_{\rho, h}^{n} \rb \big \Vert_{L^{2}(F)}^{2} 
\le 
\sum_{j=2}^6  |E_j^n(e_{\rho, h}^{n})| .
\end{align}
In the following, we estimate $|E_j^n(e_{\rho, h}^{n})|$ for $j=2,\dots,6$. 

We notice that $\nabla\cdot \left( P_h^{\rm RT}u_{h}^{n-1} \right) = 0$ in $\Omega$, 
and $ P_h^{\rm RT}u_{h}^{n-1}  \in \text{RT}^{1}(\CT_{h})$.  Thus we have 
$P_h^{\rm RT}u_{h}^{n-1} \in P^{1}(\CT_{h})^{d}$.  Then by the definition of $P_h^{\rm dG}\rho^{n}$, 
we have 
\begin{align} \label{E2n-e-rho} 
|E_2^n(e_{\rho, h}^{n})| = \vert  (P_h^{\rm RT}u_{h}^{n-1} 
 (\rho^{n} - P_h^{\rm dG}\rho^{n}), \nabla e_{\rho, h}^{n}) \vert = 0.
\end{align}

Since the value of $P_h^{\rm RT} u_{h}^{n-1} 
\cdot \nu_{F} \, (\rho^{n} - \widehat{P_h^{\rm dG}\rho^{n}}) $ on a face $F\subset\partial K$ is independent of the tetrahedron containing the face $F$, and $P_h^{\rm RT} u_{h}^{n-1} 
\cdot \nu_{F} \, (\rho^{n} - \widehat{P_h^{\rm dG}\rho^{n}}) =0$ on the boundary faces, 
it follows that 
\begin{align}
\label{interface_term_bound}
&\vert E_3^n(e_{\rho,h}^n) \vert \notag \\
\nonumber 
&= \bigg|\sum_{F \in \mathcal{F}_{h}^{I}} \big\langle   
P_h^{\rm RT} u_{h}^{n-1} 
\cdot \nu_{F} \, 
(\rho^{n} - \widehat{P_h^{\rm dG}\rho^{n}}) , \lb e_{\rho,h}^n  \rb \cdot \nu_{F} \big\rangle_{F} \bigg| \\
& \le  \sum_{F \in \mathcal{F}_{h}^{I}}  \Big\Vert  \vert  P_h^{\rm RT} u_{h}^{n-1} 
\cdot  \nu_{F}\vert^{\frac{1}{2}} (\rho^{n} - \widehat{P_h^{\rm dG}\rho^{n}} ) 
\Big\Vert_{L^{2}(F)}^{2}
+
\frac14\sum_{F \in \mathcal{F}_{h}^{I}} \Big\Vert \vert P_h^{\rm RT} u_{h}^{n-1}
\cdot  \Vn_{F}\vert^{\frac{1}{2}} \lb e_{\rho,h}^n \rb\Big\Vert_{L^{2}(F)}^{2} \nonumber \\
& \le  \Vert P_h^{\rm RT} u_{h}^{n-1}\Vert_{L^{\infty}(\Omega)}
\sum_{F \in \mathcal{F}_{h}^{I}} \Vert \rho^{n} - \widehat{P_h^{\rm dG}\rho^{n}} \Vert_{L^{2}(F)}^{2}  
+
\frac14\sum_{F \in \mathcal{F}_{h}^{I}} \Big\Vert \vert P_h^{\rm RT} u_{h}^{n-1}
\cdot  \Vn_{F}\vert^{\frac{1}{2}} \lb e_{\rho,h}^n \rb\Big\Vert_{L^{2}(F)}^{2} \nonumber \\
& \le  C
\sum_{K \in \CT_h} 
\left( h^{-1} \Vert \rho^{n} - P_h^{\rm dG}\rho^{n} \Vert_{L^{2}(K)}^{2}  
+ h \Vert \rho^{n} -  P_h^{\rm dG}\rho^{n} \Vert_{H^{1}(K)}^{2}  \right) \notag \\
&\quad\, 
+
\frac14\sum_{F \in \mathcal{F}_{h}^{I}} \Big\Vert \vert P_h^{\rm RT} u_{h}^{n-1}
\cdot  \Vn_{F}\vert^{\frac{1}{2}} \lb e_{\rho,h}^n \rb\Big\Vert_{L^{2}(F)}^{2} \nonumber \\
& \le  Ch^{-1} h^{4+2\alpha} \Vert \rho^{n}\Vert_{H^{2+\alpha}(\Omega)}^{2}  
+
\frac14\sum_{F \in \mathcal{F}_{h}^{I}} \Big\Vert \vert P_h^{\rm RT} u_{h}^{n-1}
\cdot  \Vn_{F}\vert^{\frac{1}{2}} \lb e_{\rho,h}^n \rb\Big\Vert_{L^{2}(F)}^{2} ,
\end{align}
where we have the inequality $ab = (\sqrt{2} a) (b/\sqrt{2} ) \le \frac12(\sqrt{2} a)^2+\frac12(b/\sqrt{2} )^2= a^2+b^2/4$.

For $\alpha\in(0,\frac12)$, by using the Sobolev embedding $H^{2+\alpha}(\Omega)\hookrightarrow W^{1,\frac{6}{1-2\alpha}}(\Omega)$ and $H^{2}(\Omega)\hookrightarrow W^{\frac{3}{2}+\alpha, \frac{3}{1+\alpha}}(\Omega)$ (cf. \cite[Theorem 7.43]{AdamsFournier2003}), we have 
\begin{align}
|E_4^n(e_{\rho, h}^{n})| & =  \vert  ((\Vu^{n-1} - P_h^{\rm RT}\Vu^{n-1})\cdot 
\nabla \rho^{n}, e_{\rho, h}^{n}) \vert  \nonumber \\
 & \le \Vert \Vu^{n-1} - P_h^{\rm RT}\Vu^{n-1}\Vert_{L^{\frac{3}{1+\alpha}}(\Omega)} 
\Vert \nabla \rho^{n} \Vert_{L^{\frac{6}{1-2\alpha}}(\Omega)} 
\Vert e_{\rho, h}^{n} \Vert_{L^{2}(\Omega)} 
\quad\mbox{(H\"older's inequality)} \nonumber \\
& \le C \Big( \Vert u^{n-1} - I_h u^{n-1} \Vert_{L^{\frac{3}{1+\alpha}}(\Omega)}
+ \Vert I_h u^{n-1}  - P_h^{\rm RT}
\Vu^{n-1}\Vert_{L^{\frac{3}{1+\alpha}}(\Omega)}  \Big) 
\Vert \rho^{n}\Vert_{H^{2+\alpha}(\Omega)} \Vert e_{\rho, h}^{n} \Vert_{L^{2}(\Omega)} \nonumber \\ 
& \le C \Big( \Vert u^{n-1} - I_h u^{n-1}\Vert_{L^{\frac{3}{1+\alpha}}(\Omega)} 
+ h^{-\frac{1}{2}+ \alpha}\Vert I_{h}\Vu^{n-1} - P_h^{\rm RT}
\Vu^{n-1}\Vert_{L^{2}(\Omega)} \Big) 
\Vert e_{\rho, h}^{n} \Vert_{L^{2}(\Omega)} \nonumber \\ 
& \le C h^{\frac{3}{2}+\alpha} \Big( 
 \Vert \Vu^{n-1}\Vert_{W^{\frac{3}{2}+\alpha,\frac{3}{1+\alpha}}(\Omega)} 
+ \Vert \Vu^{n-1}\Vert_{H^{2}(\Omega)} \Big)
\Vert e_{\rho, h}^{n} \Vert_{L^{2}(\Omega)} 
\qquad \mbox{(by using \eqref{postprocessing_appr})} \nonumber \\
& \le C \epsilon^{-1} h^{3+2\alpha} + \epsilon \Vert e_{\rho, h}^{n} \Vert_{L^{2}(\Omega)}^{2} , \\[10pt]
|E_5^n(e_{\rho, h}^{n})|  & = \vert (P_h^{\rm RT}(u^{n-1} - u_{h}^{n-1}) \cdot \nabla \rho^{n}, e_{\rho, h}^{n}) \vert  \nonumber \\
& \le \Vert P_h^{\rm RT}(u^{n-1} - u_{h}^{n-1})\Vert_{L^{\frac{3}{1+\alpha}}(\Omega)}
 \Vert\nabla \rho^{n}\Vert_{L^{\frac{6}{1-2\alpha}}(\Omega)}
 \Vert e_{\rho, h}^{n} \Vert_{L^{2}(\Omega)} \nonumber \\
& \le C \Big( \Vert P_h^{\rm RT} (u^{n-1} - \widehat u_h^{n-1})
\Vert_{L^{\frac{3}{1+\alpha}}(\Omega)}
+ \Vert P_h^{\rm RT} e_{u,h}^{n-1}\Vert_{L^{\frac{3}{1+\alpha}}(\Omega)} 
\Big) 
\Vert \rho^{n}\Vert_{H^{2+\alpha}(\Omega)} \Vert e_{\rho, h}^{n} \Vert_{L^{2}(\Omega)} \nonumber \\
& \le C \left(
h^{\frac{3}{2}+\alpha}(\Vert u^{n-1}\Vert_{H^{2}(\Omega)} 
+ \Vert p^{n-1}\Vert_{H^{1}(\Omega)} )
+ \Vert e_{u,h}^{n-1}\Vert_{H^{1}(\Omega)} \right)
\Vert e_{\rho, h}^{n} \Vert_{L^{2}(\Omega)} \qquad (\text{by } (\ref{postprocessing_lp}) ) \nonumber  \\
& \le  \epsilon  \Vert  e_{u,h}^{n-1}\Vert_{H^{1}(\Omega)}^{2} 
+ C\epsilon^{-1} \Vert e_{\rho, h}^{n} \Vert_{L^{2}(\Omega)}^{2} + C h^{3+2\alpha}, \\[10pt]
|E_6^n(e_{\rho, h}^{n})| 
& 
\le C \tau 
\big( \Vert \partial_{t}^{2} \rho  \Vert_{L^{2}(\Omega)} 
+ \Vert \partial_{t} \Vu \Vert_{L^{3}(\Omega)} \|\nabla \rho\|_{L^{6}(\Omega)}\big) \Vert e_{\rho, h}^{n} \Vert_{L^{2}(\Omega)} \nonumber \\
& \le C \tau^{2} + C \Vert e_{\rho, h}^{n} \Vert_{L^{2}(\Omega)}^{2}. 
\label{E7n-e-rho}
\end{align}

Substituting \eqref{E2n-e-rho}--\eqref{E7n-e-rho} into \eqref{Error-ineq-e-rho}, we obtain for $1\le n \le m$, 
\begin{align*}
& D_{\tau} \Vert e_{\rho, h}^{n} \Vert_{L^{2}(\Omega)}^{2} 
+ \frac14 \sum_{F\in \mathcal{F}_{h}^{I}} \Vert 
\vert (P_h^{\rm RT}u_{h}^{n-1})\cdot \Vn_{F}\vert^{\frac{1}{2}} 
\lb e_{\rho, h}^{n} \rb \Vert_{L^{2}(F)}^{2} \\
& 
\le C\epsilon^{-1} \left( \tau^{2} + h^{3+2\alpha}\right) 
+ \epsilon \Vert  e_{u,h}^{n-1}\Vert_{H^{1}(\Omega)}^{2} 
+ C \epsilon^{-1} \Vert e_{\rho, h}^{n} \Vert_{L^{2}(\Omega)}^{2}. 
\end{align*}
By choosing $\epsilon = 1$ and applying Gr{\"{o}}nwall's inequality, we have 
\begin{align}
\label{rho_true_error}
\nonumber
& \max_{1\le n \le m} \Vert e_{\rho, h}^{n} \Vert_{L^{2}(\Omega)}^{2} 
+ \sum_{n=1}^{m} \tau 
\sum_{F\in \mathcal{F}_{h}^{I}} \Vert 
\vert (P_h^{\rm RT}u_{h}^{n-1})\cdot \Vn_{F}\vert^{\frac{1}{2}} 
\lb e_{\rho, h}^{n} \rb \Vert_{L^{2}(F)}^{2}  \\ 
\nonumber
& 
\le 
C \Vert e_{\rho, h}^{0} \Vert_{L^{2}(\Omega)}^{2}
+ C \big( \tau^{2} + h^{3+2\alpha} \big) 
+ C \sum_{n=1}^{m} \tau \Vert  e_{u,h}^{n-1}\Vert_{H^{1}(\Omega)}^{2} \\ 
& 
\le  C \big( \tau^{2} + h^{3+2\alpha} \big) 
+ C \sum_{n=1}^{m} \tau \Vert  e_{u,h}^{n-1}\Vert_{H^{1}(\Omega)}^{2} .
\end{align}
By the last inequality and the induction assumption (\ref{pre_step_error4}), we have 
\begin{subequations}
\label{rho_errors}
\begin{align}
\label{rho_error1}
 \max_{1\le n \le m} \Vert e_{\rho, h}^{n} \Vert_{L^{2}(\Omega)} 
&\le C \big(\tau + h^{\frac{3}{2}+\alpha} + (\kappa^{\frac12} +h^{\frac{\alpha}{2}})h^{\frac{3}{2}}\big) , \\
\label{rho_error2}
 \max_{1\le n \le m} \Vert e_{\rho, h}^{n} \Vert_{L^{\infty}(\Omega)} 
&\le C h^{-\frac{d}{2}}  \max_{1\le n \le m} \Vert e_{\rho, h}^{n} \Vert_{L^{2}(\Omega)} \notag \\
&
\le C  \big(h^{-\frac{d}{2}}\tau + h^{\alpha} + \kappa^{\frac12} +h^{\frac{\alpha}{2}} \big).
\end{align}
\end{subequations}
Since all the constants $C$ above are independent of $\kappa$, by choosing a sufficiently small $\kappa$ the inequality \eqref{rho_error2} implies 
\begin{align}
\label{proved_now_step_error1}
\max_{1\le n \le m} \Vert e_{\rho, h}^{n} \Vert_{L^{\infty}(\Omega)} \le \frac{1}{4} 
\rho_{\rm min} 
\end{align}
when
\begin{align}\label{grid-ratio-1}
\tau\le \kappa h^{\frac{d}{2}} 
\,\,\,\mbox{and}\,\,\,
\mbox{$h$ is sufficiently small}.
\end{align}
In this case, 
\begin{align}
\label{rho_inf_sup_now}
\frac{1}{2}\rho_{\rm min} \le 
 \rho_{h}^{m}(x) 
\le \frac{3}{2}\rho_{\rm max}.
\end{align} 
As a result,  
\begin{align}\label{chi-rho_h-n}
\chi(\rho_h^n)=\rho_h^n\quad\mbox{for}\,\,\, 0\le n\le m . 
\end{align} 
From now on we will remove the cut-off function $\chi$ on $\rho_h^n$. 

\subsection{Estimates for $D_{\tau} e_{\rho, h}^{n}$}

We estimate $\Vert D_{\tau} e_{\rho, h}^{n} \Vert_{L^{2}(\Omega)}$ by substituting $\varphi_{h} 
= D_{\tau} e_{\rho, h}^{n}$ in (\ref{DNS_error_eq1}). Since 
\begin{align*}
E_1^n (D_{\tau} e_{\rho, h}^{n}) 
=
(D_{\tau}\rho_{\tau}^{n} - P_h^{\rm dG}(D_{\tau}\rho_{\tau}^{n}), D_{\tau} e_{\rho, h}^{n}) = 0, 
\end{align*}
we have 
\begin{align} \label{Error-Dtau-e-rho-1}
\Vert D_{\tau}e_{\rho, h}^{n} \Vert_{L^{2}(\Omega)}^{2} 
&\le
\sum_{j=2}^6 |E_j^n(D_\tau e_{\rho,h}^n)|
\nonumber  \\
&\quad 
+\vert  ((P_h^{\rm RT}u_{h}^{n-1})\cdot \nabla e_{\rho, h}^{n}, 
D_{\tau}e_{\rho, h}^{n}) \vert  + \Big| \sum_{K \in \CT_{h}} \langle  P_h^{\rm RT} u_{h}^{n-1}  \cdot 
\lb e_{\rho,h}^{n} \rb, D_{\tau} e_{\rho, h}^{n} \rangle_{\partial K_{-}^{n}} \Big| 
\nonumber \\
&=:
\sum_{j=2}^8 E_j^n(D_\tau e_{\rho,h}^n) .
\end{align}
By using the inverse and trace inequalities, we have
\begin{align*}
|E_2^n(D_\tau e_{\rho,h}^n)| 
& \le \Vert P_h^{\rm RT} u_{h}^{n-1} \Vert_{L^{\infty}(\Omega)} 
\Vert \rho^{n} - P_h^{\rm dG}\rho^{n}\Vert_{L^{2}(\Omega)} 
 Ch^{-1} \Vert D_\tau e_{\rho,h}^n \Vert_{L^{2}(\Omega)}  \\
& \le C h^{-1} \Vert \rho^{n} - P_h^{\rm dG}\rho^{n}\Vert_{L^{2}(\Omega)} 
\Vert D_\tau e_{\rho,h}^n \Vert_{L^{2}(\Omega)} \\
& \le C h \Vert \rho^{n} \Vert_{H^{2}(\Omega)} 
\Vert D_\tau e_{\rho,h}^n \Vert_{L^{2}(\Omega)}  ,\\[5pt]
|E_3^n(D_\tau e_{\rho,h}^n)| 
& \le \Vert P_h^{\rm RT} u_{h}^{n-1} \Vert_{L^{\infty}(\Omega)} 
\left(C h^{-\frac12}  \Vert \rho^{n} - P_h^{\rm dG}\rho^{n}\Vert_{L^{2}(\Omega)} 
+ h^{\frac12}  \Vert \nabla (\rho^{n} - P_h^{\rm dG}\rho^{n}) \Vert_{L^{2}(\Omega)}\right) \\
&\quad\,\,\cdot 
C h^{-\frac12} \Vert D_\tau e_{\rho,h}^n \Vert_{L^{2}(\Omega)} \\
&\le C h^{-1} 
\left( \Vert \rho^{n} - P_h^{\rm dG}\rho^{n}\Vert_{L^{2}(\Omega)} 
+ h \Vert \nabla (\rho^{n} - P_h^{\rm dG}\rho^{n}) \Vert_{L^{2}(\CT_{h})}  \right)  
\Vert D_\tau e_{\rho,h}^n \Vert_{L^{2}(\Omega)} \\
& \le C h \Vert \rho^{n} \Vert_{H^{2}(\Omega)} \Vert D_\tau e_{\rho,h}^n \Vert_{L^{2}(\Omega)} ,\\[5pt]
|E_6^n(D_\tau e_{\rho,h}^n)| 
& \le 
\big( \Vert D_{\tau}\rho^{n} - \partial_{t} \rho^{n} \Vert_{L^{2}(\Omega)} 
+ \Vert (\Vu^{n-1} -  u^{n})\cdot \nabla \rho^{n} \Vert_{L^{2}(\Omega)} \big)
\Vert D_\tau e_{\rho,h}^n \Vert_{L^{2}(\Omega)} \\
&\le 
C\tau \big(\|\partial_t^2\rho\|_{L^2(\Omega)}
+\|\partial_t u\|_{L^2(\Omega)}
\|\rho\|_{H^2(\Omega)} \big) \Vert D_\tau e_{\rho,h}^n \Vert_{L^{2}(\Omega)} , \\[10pt]
|E_7^n(D_\tau e_{\rho,h}^n)| 
& \le Ch^{-1} \Vert P_h^{\rm RT} u_{h}^{n-1} \Vert_{L^{\infty}(\Omega)} 
\Vert e_{\rho,h}^{n}\Vert_{L^{2}(\Omega)} \Vert D_\tau e_{\rho,h}^n \Vert_{L^{2}(\Omega)} \\
&\le 
Ch^{-1} 
\Vert e_{\rho,h}^{n}\Vert_{L^{2}(\Omega)} 
\Vert D_\tau e_{\rho,h}^n \Vert_{L^{2}(\Omega)}  , \\[5pt]
|E_8^n(D_\tau e_{\rho,h}^n)| 
& \le \Vert P_h^{\rm RT} u_{h}^{n-1} \Vert_{L^{\infty}(\Omega)} 
Ch^{-\frac12} 
\Vert e_{\rho,h}^{n}\Vert_{L^{2}(\Omega)} 
Ch^{-\frac12} 
\Vert D_\tau e_{\rho,h}^n \Vert_{L^{2}(\Omega)} \\
& \le 
Ch^{-1} 
\Vert e_{\rho,h}^{n}\Vert_{L^{2}(\Omega)} 
\Vert D_\tau e_{\rho,h}^n \Vert_{L^{2}(\Omega)} , 
\end{align*}
and
\begin{align*}
&|E_4^n(D_\tau e_{\rho,h}^n)| + |E_5^n(D_\tau e_{\rho,h}^n)| \\
& \le C \Vert \nabla\rho^{n}\Vert_{L^6(\Omega)} \left( 
\Vert \Vu^{n-1} - P_h^{\rm RT}\Vu^{n-1}\Vert_{L^{3}(\Omega)} 
+ \Vert P_h^{\rm RT}(\Vu^{n-1} - u_{h}^{n-1})\Vert_{L^{3}(\Omega)} \right) \Vert D_\tau e_{\rho,h}^n \Vert_{L^{2}(\Omega)} \\
& \le C \left( 
\Vert \Vu^{n-1} - P_h^{\rm RT}\Vu^{n-1}\Vert_{L^{3}(\Omega)} 
+ \Vert P_h^{\rm RT}(u^{n-1}-\widehat u_h^{n-1})\Vert_{L^{3}(\Omega)} 
+ \Vert P_h^{\rm RT} e_{u,h}^{n-1} \Vert_{L^{3}(\Omega)} \right) \Vert D_\tau e_{\rho,h}^n \Vert_{L^{2}(\Omega)} \\ 
& \le 
C \big( h^{2-\frac{d}{6}}
+h^{-\frac{d}{6}} \Vert P_h^{\rm RT} e_{u,h}^{n-1} \Vert_{L^2(\Omega)} \big)
\Vert D_\tau e_{\rho,h}^n \Vert_{L^{2}(\Omega)} \\
& \le 
C \big( h^{2-\frac{d}{6}}
+h^{-\frac{d}{6}} \Vert e_{u,h}^{n-1} \Vert_{L^2(\Omega)} \big)
\Vert D_\tau e_{\rho,h}^n \Vert_{L^{2}(\Omega)} . 
\end{align*}
Substituting the estimates of $E_j^n(D_\tau e_{\rho,h}^n)$, $j=2,\dots,8$, into \eqref{Error-Dtau-e-rho-1}, we obtain 
\begin{align*} 
\Vert D_{\tau}e_{\rho, h}^{n} \Vert_{L^{2}(\Omega)}  
\le 
C h^{-1} \big( \Vert e_{\rho,h}^{n} \Vert_{L^{2}(\Omega)}
+ \Vert  e_{u,h}^{n-1} \Vert_{L^{2}(\Omega)} \big) + C(\tau + h).
\end{align*}
By using \eqref{pre_step_error2} and \eqref{rho_error1},  we have 
\begin{align}
\label{D_t_rho_error}
\Vert D_{\tau} e_{\rho, h}^{n} \Vert_{L^{2}(\Omega)}  
\le C h^{-1}\big(\tau+ h^{\frac32+\alpha} +\kappa^{\frac12} h^{\frac32}
+h^{\frac{3}{2}+ \frac{\alpha}{2}} + \tau^{\frac{5}{6}} \big) . 
\end{align}

The following estimate of $D_{\tau}(\rho^{n} - \rho_{h}^{n})$ is also needed in our error estimation for velocity. 
\begin{lemma}
\label{lemma_dtau_rho_h_minus}
The following inequality holds:
\begin{align*}
& \vert  (D_{\tau} e_{\rho,h}^{n}, \varphi_{h}) \vert  
\le C \Vert \varphi_{h} \Vert_{H^{1}(\CT_{h})} 
(\Vert e_{\rho, h}^{n} \Vert_{L^{2}(\Omega)}
 + \Vert  e_{u,h}^{n-1}\Vert_{L^{2}(\Omega)} + \tau  + h^{2}\big)
\quad\forall\, \varphi_{h} \in {\rm P}^2_{\rm dG}(\CT_{h}) .
\end{align*}
\end{lemma}

\begin{proof}
According to (\ref{DNS_error_eq1}), we have 
\begin{align}\label{Dtau-rho-rhoh-H-1}
(D_{\tau}(\rho^{n} - \rho_{h}^{n}), \varphi_{h}) 
&=\sum_{j=2}^6 E_j^n(\varphi_h) 
-(P_h^{\rm RT}u_{h}^{n-1}\cdot \nabla e_{\rho, h}^{n}, \varphi_{h})  + \sum_{K \in \CT_{h}} \langle  P_h^{\rm RT} u_{h}^{n-1}  \cdot 
\lb e_{\rho,h}^{n} \rb, \varphi_{h} \rangle_{\partial K_{-}^{n}} \nonumber \\ 
&=\sum_{j=2}^6 E_j^n(\varphi_h) 
+(P_h^{\rm RT}u_{h}^{n-1}\cdot \nabla \varphi_{h} , e_{\rho, h}^{n} )  
- \sum_{K \in \CT_{h}} \langle  P_h^{\rm RT} u_{h}^{n-1} \cdot\nu_K\, 
\widehat e_{\rho,h}^{n} , \varphi_{h} \rangle_{\partial K}  ,
\end{align}
where we have used integration by parts, and $\widehat e_{\rho,h}^{n}$ denotes the value of $e_{\rho,h}^{n}$ from the influx side on a face $F\subset\partial K$. 

By the definition of the $L^2$-projection $P_h^{\rm dG}$, we have 
\begin{align} \label{E2n-e-rho-2} 
|E_2^n(\varphi_h)| 
& = |(P_h^{\rm RT}u_{h}^{n-1}  (\rho^{n} - P_h^{\rm dG}\rho^{n}),\nabla\varphi_h)| \nonumber \\
& \le C \Vert P_h^{\rm RT} u_h^{n-1}\Vert_{L^{\infty}(\Omega)} 
\Vert \rho^{n} - P_h^{\rm dG}\rho^{n}\Vert_{L^{2}(\Omega)} 
\Vert \varphi_h \Vert_{H^1(\CT_h)} \nonumber \\
& \le C h^2\Vert \rho^{n} \Vert_{H^{2}(\Omega)} 
\Vert \varphi_h \Vert_{H^1(\CT_h)} .
\end{align}

Similarly as the estimates in \eqref{interface_term_bound}, we have 
\begin{align} \label{interface_term_bound-2}
|E_3^n(\varphi_h)|
& = \Big\vert \sum_{F \in \mathcal{F}_{h}^{I}} \big\langle   
P_h^{\rm RT} u_{h}^{n-1}  
(\rho^{n} - \widehat{P_h^{\rm dG}\rho^{n}}) , \lb \varphi_{h} \rb \big\rangle_{F} \Big\vert \nonumber \\
\nonumber 
& \le \left( \sum_{F \in \mathcal{F}_{h}^{I}}  h_F \Vert  P_h^{\rm RT} u_{h}^{n-1}  (\rho^{n} - \widehat{P_h^{\rm dG}\rho^{n}}) 
\Vert_{L^{2}(F)}^{2} \right)^{\frac{1}{2}} \cdot 
\left( \sum_{F \in \mathcal{F}_{h}^{I}} h_{F}^{-1}\Vert 
\lb \varphi_{h} \rb \Vert_{L^{2}(F)}^{2} \right)^{\frac{1}{2}}\\
\nonumber
& \le C\Big( \sum_{K \in \CT_h}  
\Vert P_h^{\rm RT} u_{h}^{n-1} \Vert_{L^\infty(K)}^2 
(\Vert   \rho^{n} - P_h^{\rm dG}\rho^{n} 
\Vert_{L^{2}(K)}^{2} + h^{2}\Vert   \rho^{n} - P_h^{\rm dG}\rho^{n} 
\Vert_{H^{1}(K)}^{2} ) \Big)^{\frac{1}{2}} \\ 
\nonumber
& \qquad \qquad  \cdot  \Big( \sum_{F \in \mathcal{F}_{h}^{I}} h_{F}^{-1}\Vert 
\lb \varphi_{h} \rb \Vert_{L^{2}(F)}^{2}  \Big)^{\frac{1}{2}} \nonumber  \\
& \le C 
(\Vert \rho^{n} - P_h^{\rm dG}\rho^{n}
\Vert_{L^{2}(\Omega)} 
+h\Vert \rho^{n} - P_h^{\rm dG}\rho^{n}
\Vert_{H^{1}(\Omega)}) \cdot 
 \Big( \sum_{F \in \mathcal{F}_{h}^{I}} h_{F}^{-1}\Vert 
\lb \varphi_{h} \rb \Vert_{L^{2}(F)}^{2}  \Big)^{\frac{1}{2}} \nonumber \\
& \le Ch^2 \Vert \rho^{n} \Vert_{H^{2}(\Omega)}  \Vert 
\varphi_{h} \Vert_{H^1(\CT_h)} , 
\end{align}
where the last inequality uses definition \eqref{H1_discrete} of the norm $\Vert \varphi_{h} \Vert_{H^1(\CT_h)}$. 

By using integration by parts in $E_4^n(e_{\rho, h}^{n})$ and $E_5^n(e_{\rho, h}^{n})$, we have 
\begin{align}
|E_4^n(e_{\rho, h}^{n})| 
& =  \vert  ((u^{n-1} - P_h^{\rm RT}u^{n-1})\cdot 
\nabla \rho^{n}, \varphi_h) \vert  \nonumber \\
& =  \Big\vert  - ( \rho^{n}, (u^{n-1} - P_h^{\rm RT}u^{n-1})\cdot 
\nabla  \varphi_h) 
+ \sum_{F \in \mathcal{F}_{h}^{I}}((u^{n-1} - P_h^{\rm RT}u^{n-1})\,\rho^{n}, \lb \varphi_h \rb) \Big\vert  \nonumber \\
& \le \|\rho^n\|_{L^\infty(\Omega)} \Vert u^{n-1} - P_h^{\rm RT} u^{n-1}\Vert_{L^2(\Omega)} 
\Vert \nabla \varphi_h \Vert_{L^2(\Omega)} \nonumber \\
&\,\quad
+   
\|\rho^n\|_{L^\infty(\Omega)}
\bigg(\sum_{F \in \mathcal{F}_{h}^{I}} h_F \Vert \Vu^{n-1} - P_h^{\rm RT}u^{n-1}\Vert_{L^2(F)}^2\bigg)^{\frac12}
\bigg(\sum_{F \in \mathcal{F}_{h}^{I}} h_F^{-1}\Vert \lb \varphi_h \rb \Vert_{L^2(F)}^2\bigg)^{\frac12} \nonumber \\
& \le 
C 
(\Vert u^{n-1} - P_h^{\rm RT} u^{n-1}\Vert_{L^2(\Omega)} 
+ h \Vert u^{n-1} - P_h^{\rm RT} u^{n-1}\Vert_{H^1(\Omega)} )
\Vert \varphi_h \Vert_{H^1(\CT_h)}  \nonumber \\
& \le 
C h^2 \Vert u^{n-1} \Vert_{H^2(\Omega)} 
\Vert \varphi_h \Vert_{H^1(\CT_h)} , \\[10pt]
|E_5^n(e_{\rho, h}^{n})| 
& =  \vert  ((P_h^{\rm RT}u^{n-1} - P_h^{\rm RT}u_h^{n-1})\cdot 
\nabla \rho^{n}, \varphi_h) \vert  \nonumber \\
& =  \Big\vert  - ( \rho^{n}, (P_h^{\rm RT}u^{n-1} - P_h^{\rm RT}u_h^{n-1})\cdot 
\nabla  \varphi_h) 
+ \sum_{F \in \mathcal{F}_{h}^{I}}((P_h^{\rm RT}u^{n-1} - P_h^{\rm RT}u_h^{n-1})\,\rho^{n}, \lb \varphi_h \rb) \Big\vert  \nonumber \\
& \le \|\rho^n\|_{L^\infty(\Omega)} \Vert P_h^{\rm RT}u^{n-1} - P_h^{\rm RT}u_h^{n-1} \Vert_{L^2(\Omega)} 
\Vert \nabla \varphi_h \Vert_{L^2(\Omega)} \nonumber \\
&\,\quad
+   
\|\rho^n\|_{L^\infty(\Omega)}
\bigg(\sum_{F \in \mathcal{F}_{h}^{I}} h_F \Vert P_h^{\rm RT}u^{n-1} - P_h^{\rm RT}u_h^{n-1}\Vert_{L^2(F)}^2\bigg)^{\frac12}
\bigg(\sum_{F \in \mathcal{F}_{h}^{I}} h_F^{-1}\Vert \lb \varphi_h \rb \Vert_{L^2(F)}^2\bigg)^{\frac12} \nonumber \\
& \le 
C \Vert P_h^{\rm RT}u^{n-1} - P_h^{\rm RT}u_h^{n-1} \Vert_{L^2(\Omega)}  
\Vert \varphi_h \Vert_{H^1(\CT_h)} 
\quad\mbox{(inverse trace inequality)}
\nonumber \\
& \le 
C \big( \Vert P_h^{\rm RT}(u^{n-1} -\widehat u_h^{n-1}) \Vert_{L^2(\Omega)} 
+ \Vert e_{u,h}^{n-1} \Vert_{L^2(\Omega)} \big) 
\Vert \varphi_h \Vert_{H^1(\CT_h)} 
\quad\mbox{(triangle inequality)}\nonumber \\
& \le 
C \big( h^2  
+ \Vert e_{u,h}^{n-1} \Vert_{L^2(\Omega)} \big) 
\Vert \varphi_h \Vert_{H^1(\CT_h)} .
\quad\mbox{(inequality \eqref{assumption_stokes_proj} is used)} 
\end{align}

The term $E_6^n(e_{\rho, h}^{n})$ can be estimated in the same way as \eqref{E7n-e-rho}, i.e., 
\begin{align}
|E_6^n(e_{\rho, h}^{n})| 
& \le 
C \tau 
\big( \Vert \partial_{t}^{2} \rho  \Vert_{L^{2}(\Omega)} 
+ \Vert \partial_{t} \Vu \Vert_{L^{3}(\Omega)} \|\nabla \rho\|_{L^{6}(\Omega)}\big) \Vert e_{\rho, h}^{n} \Vert_{L^{2}(\Omega)}  
\le
C \tau \Vert e_{\rho, h}^{n} \Vert_{L^{2}(\Omega)} . 
\label{E7n-e-rho-2}
\end{align}

The last two terms in \eqref{Dtau-rho-rhoh-H-1} can be estimated by 
\begin{align}\label{E7-e-rho-H-1}
&\big|(P_h^{\rm RT}u_{h}^{n-1}\cdot \nabla \varphi_{h} , e_{\rho, h}^{n} ) \big|
\le
C \|e_{\rho, h}^{n}\|_{L^2(\Omega)}
\|\varphi_{h}\|_{H^1(\CT_h)} 
\end{align}
and
\begin{align}\label{E8-e-rho-H-1}
&\Big|\sum_{K \in \CT_{h}} \langle  P_h^{\rm RT} u_{h}^{n-1} \cdot\nu_K\, 
\widehat e_{\rho,h}^{n} , \varphi_{h} \rangle_{\partial K} \Big| \nonumber \\
& = \Big|\sum_{F \in \mathcal{F}_{h}^{I}} \langle  P_h^{\rm RT} u_{h}^{n-1}  \, 
\widehat e_{\rho,h}^{n} , \lb \varphi_{h} \rb \rangle_{F} \Big| \nonumber \\
& \le 
\Vert P_h^{\rm RT} u_{h}^{n-1} \Vert_{L^\infty(\Omega)}
\bigg(\sum_{F \in \mathcal{F}_{h}^{I}} 
h_F\Vert \widehat e_{\rho,h}^{n} \Vert_{L^2(F)}^2 \bigg)^{\frac12}
\bigg(\sum_{F \in \mathcal{F}_{h}^{I}} 
h_F^{-1}\Vert \lb \varphi_{h} \rb \Vert_{L^2(F)}^2 \bigg)^{\frac12} \nonumber \\
& \le 
C
\bigg(\sum_{K \in \CT_h} 
\Vert e_{\rho,h}^{n} \Vert_{L^2(K)}^2 \bigg)^{\frac12}
\bigg(\sum_{F \in \mathcal{F}_{h}^{I}} 
h_F^{-1}\Vert \lb \varphi_{h} \rb \Vert_{L^2(F)}^2 \bigg)^{\frac12}
\quad\mbox{(inverse trace inequality)} \nonumber \\
& \le 
C \Vert e_{\rho,h}^{n} \Vert_{L^2(\Omega)}
\Vert \varphi_{h} \Vert_{H^1(\CT_h)} .
\end{align}

Substituting \eqref{E2n-e-rho-2}--\eqref{E8-e-rho-H-1} into \eqref{Dtau-rho-rhoh-H-1} yields the desired result of Lemma \ref{lemma_dtau_rho_h_minus}. 
\end{proof}

\subsection{Estimates for $ e_{u,h}^{n}$}

From (\ref{DNS_eq2}) one can see that the exact solution $u^n$ satisfies the equation
\begin{align}
\label{DNS_variational_eq2} 
& (\rho^{n-1} D_{\tau}  u^{n},  v_{h}) + \frac{1}{2}((D_{\tau}\rho^{n}) u^{n},  v_{h})  
+ \frac{1}{2} ( ( \rho^{n} \Vu^{n-1} \cdot \nabla) u^{n},  v_{h}) \\
\nonumber
& \qquad - \frac{1}{2} ( ( \rho^{n}\Vu^{n-1}\cdot \nabla)  v_{h},  u^{n}) 
 + (\mu \nabla  u^{n}, \nabla  v_{h}) - (p^{n}, \nabla\cdot  v_{h})  \\
& = (R_u^n,v_h) 
\qquad\forall\, v_{h} \in \mathring{\rm P}^{\rm 1b}(\CT_{h})^d ,
\nonumber
\end{align}
with a defect $R_u^n$, which has the following expression: 
\begin{align}\label{def-R_u^n}
R_u^n
& =
(\rho^{n-1} - \rho^{n}) D_{\tau}  u^{n}  
+ \rho^{n} (D_{\tau}  u^{n} - \partial_{t}  u^{n})  
 + \frac{1}{2}(D_{\tau}\rho^{n} - \partial_{t}\rho^{n}) u^{n} 
+ \rho^{n} (\Vu^{n-1} -  u^{n}) \cdot \nabla  u^{n} .
\end{align}
Under the regularity assumption \eqref{sol_reg}, we have 
\begin{align}\label{Estimate-R_u^n}
\|R_u^n\|_{L^2(\Omega)}\le C\tau . 
\end{align}
We also note that equation (\ref{DNS_FEM_eq2}) can be rewritten as (removing the cut-off function $\chi$ 
in view of \eqref{chi-rho_h-n}) 
\begin{align*}
& (\rho_{h}^{n-1} D_{\tau}u_{h}^{n},  v_{h}) 
+ \frac{1}{2} ( D_{\tau} \rho_{h}^{n} u_{h}^{n},  v_{h}) 
+ \frac{1}{2}( (\rho_{h}^{n} u_{h}^{n-1} \cdot \nabla) u_{h}^{n},  v_{h}) \\
&\quad\,\,\, - \frac{1}{2}( (\rho_{h}^{n} u_{h}^{n-1} \cdot \nabla)  v_{h}, u_{h}^{n}) 
+ (\mu \nabla u_{h}^{n}, \nabla  v_{h}) - (p_{h}^{n}, \nabla\cdot  v_{h}) = 0
\end{align*}
Subtracting the above equations from (\ref{DNS_variational_eq2}) yields 
\begin{align}\label{Error_Eq_p}
& \big[ ( \rho_{h}^{n-1}D_{\tau}  e_{u,h}^{n},  v_{h}) 
+ (\rho_{h}^{n-1} D_{\tau} ( u^{n} - \widehat u_h^n),  v_{h} ) \notag \\ 
&\,\quad  + ((\rho^{n-1} - \rho_{h}^{n-1})D_{\tau}  u^{n},  v_{h} ) \big] \notag \\ 
&\,\quad + \frac{1}{2} \big[ (D_{\tau} \rho_h^{n} e_{u,h}^{n},  v_{h})
+(D_{\tau}(\rho^{n}-\rho_h^{n}) e_{u,h}^{n},  v_{h}) 
+ (D_{\tau}\rho^{n} ( u^{n} - \widehat u_h^n),  v_{h} ) \big] \notag \\
&\,\quad + \frac{1}{2} (D_{\tau}(\rho^{n} - \rho_{h}^{n} ) u_{h}^{n},  v_{h} )   \notag \\ 
&\,\quad + \frac{1}{2} \big[ (  ( \rho_{h}^{n} u_{h}^{n-1} \cdot \nabla )  e_{u,h}^{n},  v_{h})
+ ( ( \rho_{h}^{n} u_{h}^{n-1}\cdot\nabla) ( u^{n}-\widehat u_h^n), v_{h}) \big] \notag \\ 
&\,\quad + \frac{1}{2} \big[ ( ( \rho_{h}^{n} (\Vu^{n-1}-\widehat u_h^{n-1})
\cdot\nabla ) u^{n}, v_{h}) 
+ ( ( \rho_{h}^{n} e_{u,h}^{n-1} \cdot \nabla ) u^{n},  v_{h} ) \big] \notag \\
&\,\quad + \frac{1}{2}( ((\rho^{n} - \rho_{h}^{n})  \Vu^{n-1} \cdot \nabla) u^{n},  v_{h}) \notag \\ 
&\,\quad - \frac{1}{2} \big[ ( ( \rho_{h}^{n} u_{h}^{n-1} \cdot \nabla ) v_{h},  e_{u,h}^{n})
+ ( ( \rho_{h}^{n} u_{h}^{n-1}\cdot\nabla ) v_{h},  u^{n}-\widehat u_h^n) \big] \notag \\ 
&\,\quad
- \frac{1}{2} \big[ (  (\rho_{h}^{n} (\Vu^{n-1}-\widehat u_h^{n-1})
\cdot\nabla ) v_{h}, u^{n}) 
+ ( ( \rho_{h}^{n} e_{u,h}^{n-1} \cdot \nabla ) v_{h},  u^{n}) \big] \notag \\
&\,\quad 
- \frac{1}{2} ( ((\rho^{n} - \rho_{h}^{n}) \Vu^{n-1} \cdot \nabla ) v_{h},  u^{n})  \notag \\       
&\,\quad + (\mu \nabla  e_{u,h}^{n}, \nabla  v_{h}) - ( e_{p,h}^{n}, \nabla\cdot  v_{h}) \notag \\
& = (R_u^n ,v_{h}) 
\qquad\forall\, v_{h} \in \mathring{\rm P}^{\rm 1b}(\CT_{h})^d . 
\end{align}
Then, substituting $v_{h} = e_{u,h}^{n}$ into the above equation and using the property $( e_{p,h}^{n}, \nabla\cdot  e_{u,h}^{n})=0$ (which is a consequence of \eqref{DNS_FEM_eq3}), we obtain the following error equation of $e_{u,h}^{n}$: 
\begin{align}
\label{DNS_error_eq2_old}
&\hspace{-9pt}
(\rho_{h}^{n-1}D_{\tau} e_{u,h}^{n},  e_{u,h}^{n}) 
+ \frac{1}{2}(D_{\tau} \rho_{h}^{n} e_{u,h}^{n},  e_{u,h}^{n}) 
+ (\mu \nabla  e_{u,h}^{n}, \nabla  e_{u,h}^{n}) 
\nonumber \\
\nonumber 
= & -(\rho_{h}^{n-1}D_{\tau} ( u^{n}-\widehat u_h^n), 
 e_{u,h}^{n}) \\
\nonumber 
 & - ((\rho^{n-1} -\rho_{h}^{n-1})D_{\tau}  u^{n},  e_{u,h}^{n}) \\
\nonumber 
 & - \frac{1}{2} \big[(D_{\tau}(\rho^{n}-\rho_{h}^{n}) e_{u,h}^{n},  e_{u,h}^{n}) 
 + (D_{\tau}\rho^{n}( u^{n}-\widehat u_h^n), 
 e_{u,h}^{n})  \big] \\
\nonumber 
& -\frac{1}{2} (D_{\tau}(\rho^{n} - \rho_{h}^{n})u_{h}^{n},  e_{u,h}^{n}) \\
\nonumber 
& -\frac{1}{2} \big[ ( (\rho_{h}^{n}u_{h}^{n-1}\cdot \nabla ) e_{u,h}^{n},  e_{u,h}^{n}) 
+ ( ( \rho_{h}^{n}u_{h}^{n-1}\cdot \nabla ) ( u^{n}-\widehat u_h^n), 
 e_{u,h}^{n}) \big] \\
\nonumber 
& -\frac{1}{2} \big[ ( ( \rho_{h}^{n}(\Vu^{n-1}-\widehat u_h^{n-1} )\cdot 
\nabla ) u^{n}, e_{u,h}^{n}) + ( (\rho_{h}^{n} e_{u,h}^{n-1}\cdot 
\nabla ) u^{n}, e_{u,h}^{n}) \big] \\ 
\nonumber 
& -\frac{1}{2} ( ( (\rho^{n} - \rho_{h}^{n})\Vu^{n-1}\cdot \nabla ) u^{n}, 
 e_{u,h}^{n})  \\
\nonumber 
& + \frac{1}{2} \big[ ( ( \rho_{h}^{n}u_{h}^{n-1}\cdot \nabla ) e_{u,h}^{n},  e_{u,h}^{n}) 
+ ( (\rho_{h}^{n}u_{h}^{n-1}\cdot \nabla ) e_{u,h}^{n},  u^{n}-\widehat u^{n} ) \big] \\ 
\nonumber 
& + \frac{1}{2} \big[ ( (\rho_{h}^{n} (\Vu^{n-1}-\widehat u_h^{n-1} ) \cdot 
\nabla ) e_{u,h}^{n},  u^{n}) + ( ( \rho_{h}^{n} e_{u,h}^{n-1}\cdot 
\nabla ) e_{u,h}^{n},  u^{n}) \big] \\
\nonumber 
& + \frac{1}{2} ( ( (\rho^{n}- \rho_{h}^{n})\Vu^{n-1}\cdot \nabla ) e_{u,h}^{n},  u^{n}) \\ 
\nonumber 
& + (R_u^n ,e_{u,h}^{n}) \\
=& \sum_{j=1}^{11} F_j^n  . 
\end{align}
Since $(\rho_{h}^{n}u_{h}^{n-1}\cdot \nabla 
 e_{u,h}^{n},  e_{u,h}^{n}) $ appears with opposite signs in both $F_5^n$ and $F_8^n$, it follows that 
 \begin{align}
\label{DNS_error_eq2}
&\hspace{-9pt}
(\rho_{h}^{n-1}D_{\tau} e_{u,h}^{n},  e_{u,h}^{n}) 
+ \frac{1}{2}(D_{\tau}\rho_{h}^{n} e_{u,h}^{n},  e_{u,h}^{n}) 
+ (\mu \nabla  e_{u,h}^{n}, \nabla  e_{u,h}^{n}) 
=
\sum_{j=1}^{11} \widehat F_j^n + (R_u^n , e_{u,h}^{n} )  . 
\end{align}
where $\widehat F_j^n=F_j^n$ for $j\neq 5,8$, and 
\begin{align}
\widehat F_5^n
&=
-\frac{1}{2} ( ( \rho_{h}^{n}u_{h}^{n-1}\cdot \nabla ) ( u^{n}-\widehat u_h^n), e_{u,h}^{n}) , \\
\widehat F_8^n
&=
\frac{1}{2} ( (\rho_{h}^{n}u_{h}^{n-1}\cdot \nabla )  e_{u,h}^{n},  u^{n}-\widehat u^{n} ) .
\end{align}
In the following, we estimate $\vert \widehat F_{j}^{n}\vert$ for $j=1,\dots,11$. 

First, we note that $(D_{\tau}\widehat u_h^n,D_{\tau}\widehat p_h^n)$ is actually the Stokes--Ritz projection of $(D_{\tau}u^n,D_{\tau}p^n)$. Therefore, \eqref{assumption_stokes_proj} implies 
\begin{align*}
\Vert D_{\tau}u^n-D_{\tau}\widehat u_h^n\Vert_{L^{2}(\Omega)} 
\le
C h^{2} \left( \Vert D_{\tau}u^n\Vert_{H^{2}(\Omega)} 
+ \Vert D_{\tau}p^n\Vert_{H^{1}(\Omega)}  \right) 
\le
C h^2 .
\end{align*}
By using this result and the property $\|\rho_{h}^{n-1}\|_{L^\infty(\Omega)}\le \frac32\rho_{\max}$, we have 
\begin{align*}
\vert \widehat F_{1}^{n}\vert = & \vert (\rho_{h}^{n-1}D_{\tau} ( u^{n}-\widehat u_h^n), 
 e_{u,h}^{n})\vert 
\le C h^2 \Vert  e_{u,h}^{n}\Vert_{L^{2}(\Omega)}  
\le C \epsilon^{-1}h^{4} 
+ \epsilon \Vert  e_{u,h}^{n} \Vert_{L^{2}(\Omega)}^{2} .
\end{align*}

Second, we have 
\begin{align*}
\vert \widehat F_{2}^{n} \vert 
= & \vert ((\rho^{n-1} - \rho_{h}^{n-1})D_{\tau}  u^{n},  e_{u,h}^{n}) \vert \\  
\le & \Vert \rho^{n-1} - \rho_{h}^{n-1}  \Vert_{L^{2}(\Omega)} \Vert D_{\tau}  u^{n}\Vert_{L^{3}(\Omega)} \Vert  e_{u,h}^{n}\Vert_{L^{6}(\Omega)} 
 \\ 
\le & C\Vert \rho^{n-1} - \rho_{h}^{n-1} \Vert_{L^{2}(\Omega)} \Vert D_{\tau}  u^{n}\Vert_{H^1(\Omega)} \Vert \nabla e_{u,h}^{n}\Vert_{L^{2}(\Omega)} 
\quad\mbox{(Sobolv embedding)}
 \\
 \le & C \big( \Vert \rho^{n-1} - P_h^{\rm dG} \rho^{n-1}\Vert_{L^{2}(\Omega)} 
 +  \Vert e_{\rho, h}^{n-1} \Vert_{L^{2}(\Omega)} \big) \Vert D_{\tau}  u^{n}\Vert_{H^1(\Omega)} \Vert \nabla e_{u,h}^{n}\Vert_{L^{2}(\Omega)} \\ 
 \le & C \big( h^{2}  
 + \Vert e_{\rho, h}^{n-1} \Vert_{L^{2}(\Omega)} \big) 
 \Vert \nabla  e_{u,h}^{n}\Vert_{L^{2}(\Omega)} \\ 
 \le & C \epsilon^{-1} \big( h^{4} + \Vert e_{\rho, h}^{n-1} \Vert_{L^{2}(\Omega)}^{2} \big) 
 + \epsilon \Vert \nabla e_{u,h}^{n} \Vert_{L^2(\Omega)}^{2}, \\[5pt] 
\vert \widehat F_{3}^{n} \vert 
= & \frac{1}{2} \vert (D_{\tau}(\rho^{n} - \rho_{h}^{n})\, e_{u,h}^{n},  e_{u,h}^{n}) 
+ (D_{\tau}\rho^{n}( u^{n}-\widehat u_h^n),  e_{u,h}^{n}) \vert \\
\le & \Vert D_{\tau}(\rho^{n} - \rho_h^{n})\Vert_{L^2(\Omega)} 
\Vert  e_{u,h}^{n}\Vert_{L^{4}(\Omega)}^{2} 
 + \Vert D_{\tau}\rho^{n}\Vert_{L^{3}(\Omega)} 
\Vert  u^{n}-\widehat u_h^n\Vert_{L^{2}(\Omega)} 
\Vert  e_{u,h}^{n}\Vert_{L^6(\Omega)} \\ 
\le & C \Vert D_{\tau}(\rho^{n} - \rho_h^{n})\Vert_{L^2(\Omega)} 
\Vert \nabla e_{u,h}^{n}\Vert_{L^2(\Omega)}^{2} 
 + \Vert D_{\tau}\rho^{n}\Vert_{H^1(\Omega)} 
\Vert  u^{n}-\widehat u_h^n\Vert_{L^{2}(\Omega)} 
\Vert \nabla e_{u,h}^{n}\Vert_{L^2(\Omega)} \\
\le & C \big( \Vert D_{\tau} e_{\rho, h}^{n} \Vert_{L^{2}}  + h^2 \|\rho^n\|_{H^2(\Omega)} \big)
\Vert \nabla e_{u,h}^{n}\Vert_{L^2(\Omega)}^{2} 
+ \Vert  u^{n}-\widehat u_h^n\Vert_{L^{2}(\Omega)} 
\Vert \nabla e_{u,h}^{n}\Vert_{L^2(\Omega)} \\
\le & 
C \big( \Vert D_{\tau} e_{\rho, h}^{n} \Vert_{L^{2}}  + h^2 
+\epsilon \big)
\Vert \nabla e_{u,h}^{n}\Vert_{L^2(\Omega)}^{2} 
+ C \epsilon^{-1} h^{4} ,\\[5pt]
\vert \widehat F_{5}^{n} \vert 
= & \frac{1}{2} \vert ( ( \rho_{h}^{n}u_{h}^{n-1}\cdot \nabla )
( u^{n}-\widehat u_h^n),  e_{u,h}^{n})\vert \\
\le & \vert ( (\rho_h^{n} (u_{h}^{n-1}-\Vu^{n-1})\cdot \nabla )
( u^{n}-\widehat u_h^n),  e_{u,h}^{n})\vert + \vert ( ( (\rho_h^{n} - \rho^{n} )\Vu^{n-1} \cdot \nabla )
( u^{n}-\widehat u_h^n),  e_{u,h}^{n})\vert \\
& + \vert ( ( \rho^{n}  (\Vu^{n-1}- u^{n}) \cdot \nabla )
( u^{n}-\widehat u_h^n),  e_{u,h}^{n})\vert + \vert ( ( \rho^{n}   u^{n} \cdot \nabla )
( u^{n}-\widehat u_h^n),  e_{u,h}^{n})\vert \\
\le & \vert ( (\rho_h^{n} (u_{h}^{n-1}-\Vu^{n-1})\cdot \nabla )
( u^{n}-\widehat u_h^n),  e_{u,h}^{n})\vert + \vert ( ( (\rho_h^{n} - \rho^{n} )\Vu^{n-1} \cdot \nabla )
( u^{n}-\widehat u_h^n),  e_{u,h}^{n})\vert \\
& + \vert ( ( \rho^{n}  (\Vu^{n-1}- u^{n}) \cdot \nabla )
( u^{n}-\widehat u_h^n),  e_{u,h}^{n})\vert \\
&+ \vert 
-( \nabla\cdot (\rho^{n}   u^{n})( u^{n}-\widehat u_h^n), 
 e_{u,h}^{n})
-( u^{n}-\widehat u_h^n , 
\rho^{n} u^{n} \cdot \nabla  e_{u,h}^{n})\vert \\
= & \vert ( (\rho_h^{n} (u_{h}^{n-1}-\Vu^{n-1})\cdot \nabla )
( u^{n}-\widehat u_h^n),  e_{u,h}^{n})\vert + \vert ( ( (\rho_h^{n} - \rho^{n} )\Vu^{n-1} \cdot \nabla )
( u^{n}-\widehat u_h^n),  e_{u,h}^{n})\vert \\
& + \vert ( (\rho^{n}  (\Vu^{n-1}- u^{n}) \cdot \nabla )
( u^{n}-\widehat u_h^n),  e_{u,h}^{n})\vert \\
&+ \vert 
( \partial_t\rho^{n} ( u^{n}-\widehat u_h^n), 
 e_{u,h}^{n})
-( u^{n}-\widehat u_h^n, 
\rho^{n} u^{n} \cdot \nabla  e_{u,h}^{n})\vert 
\quad\mbox{(use $\partial_t\rho^{n}+\nabla\cdot(\rho^{n} u^{n})=0$)} \\ 
\le & C\Vert \Vu^{n-1} - u_{h}^{n-1}\Vert_{L^{2}(\Omega)} 
\Vert  u^{n} - \widehat u_h^n \Vert_{H^{1}(\Omega)} 
\Vert  e_{u,h}^{n}\Vert_{L^{\infty}(\Omega)} \\
&
+ C  
\Vert \rho^{n} - \rho_{h}^{n}\Vert_{L^{2}(\Omega)} 
\Vert \Vu^{n-1}\Vert_{L^\infty(\Omega)} 
\Vert \nabla( u^{n} - \widehat u_h^n)\Vert_{L^2(\Omega)} 
\Vert  e_{u,h}^{n}\Vert_{L^{\infty}(\Omega)} \\
&
+ \Vert \rho^{n}\Vert_{L^{\infty}(\Omega)}
\tau \Vert \partial_{t}\Vu \Vert_{L^{\infty}(0,T; L^3(\Omega))} 
\Vert \nabla( u^{n} - \widehat u_h^n)\Vert_{L^2(\Omega)} 
\Vert  e_{u,h}^{n}\Vert_{L^6(\Omega)} \\
& + C (\Vert \partial_{t}\rho^{n}\Vert_{H^{1}(\Omega)}  
+ \Vert \rho^{n}\Vert_{L^{\infty}(\Omega)}\Vert  u^{n}\Vert_{L^{\infty}(\Omega)}) 
\Vert  u^{n} - \widehat u_h^n\Vert_{L^{2}(\Omega)} 
\Vert  e_{u,h}^{n} \Vert_{H^{1}(\Omega)} \\
\le & C h^{-\frac{1}{2}} \Vert \Vu^{n-1} - u_{h}^{n-1}\Vert_{L^{2}(\Omega)} 
\Vert  u^{n} - \widehat u_h^n \Vert_{H^{1}(\Omega)} 
\Vert \nabla e_{u,h}^{n}\Vert_{L^2(\Omega)} \quad\mbox{(inverse inequality)} \\
&+ C h^{-\frac{1}{2}} \Vert \rho^{n} - \rho_{h}^{n}\Vert_{L^{2}(\Omega)} 
\Vert  u^{n} - \widehat u_h^n\Vert_{H^{1}(\Omega)} 
\Vert \nabla e_{u,h}^{n}\Vert_{L^2(\Omega)} \\ 
&+ C \tau \Vert  u^{n} - \widehat u_h^n\Vert_{H^{1}(\Omega)} 
\Vert \nabla e_{u,h}^{n}\Vert_{L^2(\Omega)}
+ C \Vert  u^{n} -\widehat u_h^n\Vert_{L^{2}(\Omega)} 
\Vert \nabla e_{u,h}^{n}\Vert_{L^2(\Omega)} \\ 
\le & C h^{\frac{1}{2}} \big( \Vert  e_{u,h}^{n-1}\Vert_{L^{2}(\Omega)}
+ \Vert e_{\rho, h}^{n} \Vert_{L^{2}(\Omega)} + h^{2} \big) 
\Vert \nabla e_{u,h}^{n}\Vert_{L^2(\Omega)} 
 + C h (\tau + h) \Vert \nabla e_{u,h}^{n}\Vert_{L^2(\Omega)} \\
\le & \left( h + \epsilon \right)
\Vert \nabla e_{u,h}^{n}\Vert_{L^2(\Omega)}^{2} 
+ C \big( \Vert  e_{u,h}^{n-1}\Vert_{L^{2}(\Omega)}^{2} 
+ \Vert e_{\rho, h}^{n} \Vert_{L^{2}(\Omega)}^{2} \big) 
+ C \epsilon^{-1} h^{2} \left( \tau^{2}+ h^{2}\right) ,\\[5pt]
\vert \widehat F_{6}^{n} \vert = & \frac{1}{2} \vert ( (\rho_{h}^{n}(\Vu^{n-1}-\widehat u_h^{n-1} )
\cdot \nabla ) u^{n}, e_{u,h}^{n}) + ( (\rho_h^{n} e_{u,h}^{n-1}\cdot 
\nabla ) u^{n}, e_{u,h}^{n}) \vert \\
\le & \frac{3}{2} \rho_{\rm max} \big( \Vert \Vu^{n-1}
-\widehat u_h^{n-1}  \Vert_{L^{2}(\Omega)} 
+ \Vert  e_{u,h}^{n-1} \Vert_{L^{2}(\Omega)} \big) 
\Vert \nabla u^{n} \Vert_{L^3(\Omega)} \Vert  e_{u,h}^{n}\Vert_{L^{6}(\Omega)} \\
\le & C \big( h^{2} 
 + \Vert  e_{u,h}^{n-1} \Vert_{L^{2}(\Omega)} \big) 
\Vert u^{n} \Vert_{H^2(\Omega)} \Vert \nabla e_{u,h}^{n}\Vert_{L^{2}(\Omega)} \\ 
\le & C \epsilon^{-1} \big( h^{4} + \Vert  e_{u,h}^{n-1} \Vert_{L^{2}(\Omega)}^{2} \big) 
+ \epsilon \Vert \nabla e_{u,h}^{n}\Vert_{L^{2}(\Omega)}^{2}, \\[5pt]
\vert \widehat F_{7}^{n} \vert = & \frac{1}{2} \vert ( ( (\rho^{n} - \rho_h^{n})\Vu^{n-1}\cdot 
\nabla ) u^{n},  e_{u,h}^{n}) \vert \\
\le & C 
\left(\Vert \rho^{n} - P_h^{\rm dG}\rho^{n}\Vert_{L^{2}(\Omega)} 
+ \Vert e_{\rho, h}^{n}\Vert_{L^{2}(\Omega)} \right) \Vert \Vu^{n-1}\Vert_{L^{\infty}(\Omega)} 
\Vert \nabla u^{n}\Vert_{L^3(\Omega)} \Vert  e_{u,h}^{n}\Vert_{L^6(\Omega)} \\ 
\le & C \left(h^{2}\Vert \rho^{n} \Vert_{H^{2}(\Omega)}
 +  \Vert e_{\rho, h}^{n}\Vert_{L^{2}(\Omega)}\right) 
\Vert \nabla e_{u,h}^{n}\Vert_{L^2(\Omega)} \\ 
\le & \epsilon \Vert \nabla e_{u,h}^{n}\Vert_{L^2(\Omega)}^{2}
+ C \epsilon^{-1} \big( h^{4} + \Vert e_{\rho, h}^{n}\Vert_{L^{2}(\Omega)}^{2} \big), \\[5pt] 
\vert \widehat F_{8}^{n} \vert 
= & 
\frac{1}{2} \vert  (( \rho_h^{n} u_{h}^{n-1}\cdot 
\nabla ) e_{u,h}^{n},  u^{n}-\widehat u_h^n) \vert \\ 
\le & 
\frac{3}{2} \rho_{\rm max} \Vert u_{h}^{n-1}\Vert_{L^{\infty}(\Omega)} 
\Vert \nabla e_{u,h}^{n}\Vert_{L^2(\Omega)} 
\Vert  u^{n}-\widehat u_h^n\Vert_{L^{2}(\Omega)} \\ 
\le & 
C \Vert \nabla e_{u,h}^{n}\Vert_{L^2(\Omega)}
\Vert  u^{n}-\widehat u_h^n\Vert_{L^{2}(\Omega)}
 \quad (\text{by }(\ref{u_inf_sup_pre})) \\ 
\le & 
\epsilon \Vert \nabla e_{u,h}^{n}\Vert_{L^2(\Omega)}^{2}
+ C \epsilon^{-1} h^{4}, \\[5pt]
\vert \widehat F_{9}^{n} \vert = & \frac{1}{2} \vert ( ( \rho_h^{n} (\Vu^{n-1}-\widehat u_h^{n-1} ) 
\cdot \nabla ) e_{u,h}^{n},  u^{n}) 
+ ( ( \rho_{h}^{n} e_{u,h}^{n-1}\cdot \nabla )  e_{u,h}^{n},  u^{n}) \vert \\
\le & \frac{3}{2} \rho_{\rm max}  
\big( \Vert \Vu^{n-1}-\widehat u_h^{n-1}  \Vert_{L^{2}(\Omega)} 
 + \Vert  e_{u,h}^{n-1} \Vert_{L^{2}(\Omega)} \big) \Vert \nabla  e_{u,h}^{n} \Vert_{L^{2}(\Omega)} 
\Vert  u^{n}\Vert_{L^{\infty}(\Omega)} \\ 
\le & C \big( \Vert \Vu^{n-1}-\widehat u_h^{n-1}  \Vert_{L^{2}(\Omega)} 
 + \Vert  e_{u,h}^{n-1} \Vert_{L^{2}(\Omega)} \big) \Vert \nabla  e_{u,h}^{n} \Vert_{L^{2}(\Omega)} \\
\le & C \big( h^{2} + \Vert  e_{u,h}^{n-1} \Vert_{L^{2}(\Omega)} \big) 
\Vert \nabla  e_{u,h}^{n} \Vert_{L^{2}(\Omega)} \\
\le & \epsilon \Vert \nabla  e_{u,h}^{n} \Vert_{L^{2}(\Omega)}^{2} 
+ C \epsilon^{-1} \big( h^{4} + \Vert  e_{u,h}^{n-1} \Vert_{L^{2}(\Omega)}^{2} \big), \\[5pt]
\vert \widehat F_{10}^{n} \vert = & \frac{1}{2} \vert ( ( (\rho^{n}- \rho_h^{n})\Vu^{n-1}\cdot 
\nabla ) e_{u,h}^{n},  u^{n}) \vert \\
\le & \left(  \Vert \rho^{n}- P_h^{\rm dG} \rho^{n}\Vert_{L^{2}(\Omega)} 
+ \Vert e_{\rho, h}^{n}\Vert_{L^{2}(\Omega)} \right) \Vert \Vu^{n-1}\Vert_{L^{\infty}(\Omega)} 
\Vert \nabla  e_{u,h}^{n} \Vert_{L^{2}(\Omega)} \Vert  u^{n} \Vert_{L^{\infty}(\Omega)} \\ 
\le & C \big(  \Vert \rho^{n}- P_h^{\rm dG} \rho^{n}\Vert_{L^{2}(\Omega)} 
+ \Vert e_{\rho, h}^{n}\Vert_{L^{2}(\Omega)} \big) \Vert \nabla  e_{u,h}^{n} \Vert_{L^{2}(\Omega)} \\
\le & \big( h^{2}\Vert \rho^{n}\Vert_{H^{2}(\Omega)} 
 + \Vert e_{\rho, h}^{n}\Vert_{L^{2}(\Omega)} \big) 
\Vert \nabla  e_{u,h}^{n} \Vert_{L^{2}(\Omega)} \\ 
\le & \epsilon \Vert \nabla  e_{u,h}^{n} \Vert_{L^{2}(\Omega)}^{2} 
+ C \epsilon^{-1} \big( h^{4} + \Vert e_{\rho, h}^{n}\Vert_{L^{2}(\Omega)}^{2} \big).
\end{align*}
Under regularity \eqref{sol_reg}, the truncation error defined in \eqref{def-R_u^n} satisfies $\|R_{u}^n\|_{L^2(\Omega)}\le C\tau$. Therefore,    
\begin{align*}
\vert \widehat F_{11}^n \vert 
\le 
C \tau \Vert  e_{u,h}^{n} \Vert_{L^{2}(\Omega)}
\le 
C \tau^{2} + C \Vert  e_{u,h}^{n} \Vert_{L^{2}(\Omega)}^{2}.
\end{align*}

It remains to estimate $|\widehat F_{4}^{n}|$. To this end, we use the following inequality inequality and error estimate:
\begin{align}
&\|e_{u,h}^n\|_{L^\infty(\Omega)} \le Ch^{-\frac12} \|e_{u,h}^n\|_{H^1(\Omega)}
&&\mbox{for}\,\,\, d=2,3, \label{inverse-Linfty-error-1} \\
&\|u^n-\widehat u_h^n\|_{L^\infty(\Omega)} 
\le Ch^{\frac12} 
(\|u^n\|_{H^2(\Omega)}+\|p^n\|_{H^1(\Omega)})
&&\mbox{for}\,\,\, d=2,3 , \label{inverse-Linfty-error-2} \\
&\|D_\tau\rho^n - P_h^{\rm dG}D_\tau\rho^n\|_{H^{-1}(\Omega)} 
\le Ch^2
\|D_\tau\rho^n\|_{H^1(\Omega)} , \label{inverse-Linfty-error-3}
\end{align}
where the second inequality can be proved by using \eqref{assumption_stokes_proj} combined with inverse inequality and triangle inequality. 
By using Lemma \ref{lemma_dtau_rho_h_minus} and Lemma \ref{lemma_P_rho_H1}, and \eqref{inverse-Linfty-error-1}--\eqref{inverse-Linfty-error-3}, we have 
\begin{align*}
\widehat F_{4}^{n} 
= & \frac{1}{2} (D_{\tau}(\rho^{n} - \rho_{h}^{n})u_{h}^{n},  e_{u,h}^{n}) \\
= & \frac{1}{2} (D_{\tau}e_{\rho,h}^{n} , P_h^{\rm dG}( u_{h}^{n}\cdot e_{u,h}^{n}) ) 
+\frac{1}{2} (D_{\tau}(\rho^{n} - P_h^{\rm dG}\rho^{n}), u_{h}^{n}\cdot e_{u,h}^{n} )  \\
\le &
C\Vert P_h^{\rm dG}( u_{h}^{n} e_{u,h}^{n}) \Vert_{H^{1}(\CT_{h})} 
(\Vert e_{\rho, h}^{n} \Vert_{L^{2}(\Omega)}
 + \Vert  e_{u,h}^{n-1}\Vert_{L^{2}(\Omega)} + \tau  + h^{2}\big) 
 +Ch^2 \|D_{\tau} \rho^{n}\|_{H^1(\Omega)}
 \|u_{h}^{n} e_{u,h}^{n}\|_{H^1(\Omega)}
  \\
 \le &
C\Vert  u_{h}^{n} e_{u,h}^{n} \Vert_{H^{1}(\CT_{h})} 
(\Vert e_{\rho, h}^{n} \Vert_{L^{2}(\Omega)}
 + \Vert  e_{u,h}^{n-1}\Vert_{L^{2}(\Omega)} + \tau  + h^{2}\big) 
\quad\mbox{(due to Lemma \ref{lemma_P_rho_H1})} \\
= &
C\Vert \nabla(u_{h}^{n} e_{u,h}^{n} )\Vert_{L^2(\Omega)} 
(\Vert e_{\rho, h}^{n} \Vert_{L^{2}(\Omega)}
 + \Vert  e_{u,h}^{n-1}\Vert_{L^{2}(\Omega)} + \tau  + h^{2}\big) \\
\le &
C(\Vert  u_{h}^{n} \Vert_{L^\infty(\Omega)}\Vert \nabla e_{u,h}^{n} \Vert_{L^2(\Omega)} 
+\Vert \nabla u_{h}^{n} \Vert_{L^3(\Omega)}
\Vert  e_{u,h}^{n} \Vert_{L^6(\Omega)})
(\Vert e_{\rho, h}^{n} \Vert_{L^{2}(\Omega)}
 + \Vert  e_{u,h}^{n-1}\Vert_{L^{2}(\Omega)} + \tau  + h^{2}\big) \\
\le &
C(\Vert e_{u,h}^{n} \Vert_{L^\infty(\Omega)}+\Vert u^{n}-\widehat u_h^n \Vert_{L^\infty(\Omega)}
+\Vert \nabla e_{u,h}^{n} \Vert_{L^3(\Omega)}+\Vert u^{n}-\widehat u_h^n \Vert_{W^{1,3}(\Omega)}) \Vert \nabla e_{u,h}^{n} \Vert_{L^2(\Omega)}  \\
&\cdot 
(\Vert e_{\rho, h}^{n} \Vert_{L^{2}(\Omega)}
 + \Vert  e_{u,h}^{n-1}\Vert_{L^{2}(\Omega)} + \tau  + h^{2}\big) \\
\le &
C(h^{-\frac12}\Vert \nabla e_{u,h}^{n} \Vert_{L^2(\Omega)} + h^{\frac12} ) \Vert \nabla e_{u,h}^{n} \Vert_{L^2(\Omega)} 
(\Vert e_{\rho, h}^{n} \Vert_{L^{2}(\Omega)}
 + \Vert  e_{u,h}^{n-1}\Vert_{L^{2}(\Omega)} + \tau  + h^{2}\big) \\
 \le &
Ch^{-\frac12}\Vert e_{\rho, h}^{n} \Vert_{L^{2}(\Omega)} 
\Vert \nabla e_{u,h}^{n} \Vert_{L^2(\Omega)}^2  
+Ch^{\frac12} \Vert \nabla e_{u,h}^{n} \Vert_{L^2(\Omega)} \Vert e_{\rho, h}^{n} \Vert_{L^{2}(\Omega)} \\
&
+C(\Vert  e_{u,h}^{n-1}\Vert_{L^{2}(\Omega)} + \tau  + h^{2}\big) h^{-\frac12}\Vert \nabla e_{u,h}^{n} \Vert_{L^2(\Omega)}^2 \\
&
+C(\Vert  e_{u,h}^{n-1}\Vert_{L^{2}(\Omega)} + \tau  + h^{2}\big) h^{\frac12}\Vert \nabla e_{u,h}^{n} \Vert_{L^2(\Omega)} \\
 \le &
\Big[Ch^{-\frac12} \big(\tau + h^{\frac{3}{2}+\frac{\alpha}{2}} + \kappa^{\frac12} h^{\frac{3}{2}}\big)   
+Ch^{-\frac12}(h^{\frac{3}{2}+\frac{\alpha}{2}}  + \tau^{\frac56}  + \tau  + h^{2}\big) \Big]
\Vert \nabla e_{u,h}^{n} \Vert_{L^2(\Omega)}^2  \\
&
+Ch^{\frac12}  \Vert \nabla e_{u,h}^{n} \Vert_{L^2(\Omega)} ^2 
+Ch^{\frac12} \Vert e_{\rho, h}^{n} \Vert_{L^{2}(\Omega)}^2 
+Ch^{\frac12}  (\Vert  e_{u,h}^{n-1}\Vert_{L^{2}(\Omega)}^2 + \tau^2  + h^{4}\big) , 
\end{align*}
where we have used \eqref{rho_errors} and \eqref{pre_step_error2} in deriving the last inequality. With the stepsize restriction $\tau\le \kappa h^{\frac{d}{2}} $, the inequality above furthermore implies  
\begin{align*}
\widehat F_{4}^{n} 
&\le 
Ch^{\frac12} (\Vert \nabla e_{u,h}^{n} \Vert_{L^2(\Omega)}^2 
+ \Vert e_{\rho, h}^{n} \Vert_{L^{2}(\Omega)}^2 
+  \Vert  e_{u,h}^{n-1}\Vert_{L^{2}(\Omega)}^2) 
+ C( \tau^2  + h^{4}\big) . 
\end{align*}

By substituting the consistency error \eqref{Estimate-R_u^n} and the estimates of $\widehat F_{j}$, $j=1,\dots,11$, into (\ref{DNS_error_eq2}), we obtain  
\begin{align*}
& \frac{1}{2}D_{\tau} \Vert \sqrt{\rho_{h}^{n}}  e_{u,h}^{n} \Vert_{L^{2}(\Omega)}^{2} 
+ \mu\Vert  e_{u,h}^{n}\Vert_{H^{1}(\Omega)}^{2} \\
& \le 
(\epsilon+Ch^{\frac12}) \Vert \nabla e_{u,h}^{n} \Vert_{L^2(\Omega)}^2 
+ C\epsilon^{-1} (\Vert e_{\rho, h}^{n} \Vert_{L^{2}(\Omega)}^2 +\Vert e_{\rho, h}^{n-1} \Vert_{L^{2}(\Omega)}^2  
+  \Vert  e_{u,h}^{n-1}\Vert_{L^{2}(\Omega)}^2) 
+ C\epsilon^{-1} ( \tau^2  + h^{4}\big) . 
\end{align*}
By choosing sufficiently small $\epsilon$ and $h$, the first term on the right-hand side can be absorbed by the left-hand side. Then, summing up the inequality above for $n=1, \cdots, k$, we obtain for $k = 1, \cdots, m$,
\begin{align}
\label{energy_estimate_u}
& \frac{1}{2} \Vert \sqrt{\rho_h^{k}}  e_{u,h}^{k} \Vert_{L^{2}(\Omega)}^{2} 
+  \frac{\mu}{2}\sum_{n=0}^{k} \tau \Vert  e_{u,h}^{n} \Vert_{H^{1}(\Omega)}^{2} \notag \\ 
& \le 
C \sum_{n=0}^{k} \tau (\Vert e_{\rho, h}^{n} \Vert_{L^{2}(\Omega)}^2 
+\Vert e_{\rho, h}^{n-1} \Vert_{L^{2}(\Omega)}^2 +  \Vert  e_{u,h}^{n-1}\Vert_{L^{2}(\Omega)}^2) 
+ C ( \tau^2  + h^{4}\big)  .
\end{align}
Summing up $\lambda\times$\eqref{rho_true_error} and \eqref{energy_estimate_u}, we obtain 
\begin{align*}
& \lambda \Vert e_{\rho, h}^{k}\Vert_{L^{2}(\Omega)}^{2} 
+\frac{1}{2} \Vert \sqrt{\rho_h^{k}}  e_{u,h}^{k} \Vert_{L^{2}(\Omega)}^{2} 
+ \frac{\mu}{2} \sum_{n=0}^{k} \tau \Vert  e_{u,h}^{n} \Vert_{H^{1}(\Omega)}^{2} \\ 
& \le C (\tau^{2} + h^{3+2\alpha}) 
+  C \lambda \sum_{n=0}^{k}  \tau 
\Vert  e_{u,h}^{n} \Vert_{H^{1}(\Omega)}^{2} \\ 
& \quad + C \tau \sum_{n=1}^{k} 
\left( \Vert e_{\rho, h}^{n} \Vert_{L^{2}(\Omega)}^{2} 
+ \Vert e_{\rho, h}^{n-1} \Vert_{L^{2}(\Omega)}^{2} +  \Vert  e_{u,h}^{n-1}\Vert_{L^{2}(\Omega)}^{2} \right). 
\end{align*} 
By choosing $\lambda$ small enough, the term $C \lambda \sum_{n=0}^{k}  \tau 
\Vert  e_{u,h}^{n} \Vert_{H^{1}(\Omega)}^{2}$ can be absorbed by the left-hand side and we obtain for $1\le k\le m$ 
\begin{align}
\label{energy_estiamte_u_rho}
& \Vert e_{\rho, h}^{k}\Vert_{L^{2}(\Omega)}^{2} 
+ \Vert \sqrt{\rho_h^{k}}  e_{u,h}^{k} \Vert_{L^{2}(\Omega)}^{2} 
+ \sum_{n=0}^{k} \tau \Vert  e_{u,h}^{n} \Vert_{H^{1}(\Omega)}^{2} \\ 
\nonumber
& \le C (\tau^{2} + h^{3+2\alpha}) 
+ C \tau \sum_{n=1}^{k} 
\left(\Vert  e_{u,h}^{n-1}\Vert_{L^{2}(\Omega)}^{2} 
+ \Vert e_{\rho, h}^{n} \Vert_{L^{2}(\Omega)}^{2} 
+ \Vert e_{\rho, h}^{n-1} \Vert_{L^{2}(\Omega)}^{2} \right) . 
\end{align} 
Applying Gr{\"{o}}nwall's inequality to (\ref{energy_estiamte_u_rho}) and using \eqref{rho_inf_sup_pre} and \eqref{rho_inf_sup_now}, we have 
\begin{align}
\label{true_error_u_rho}
\max_{1\le n \le m} \left( \Vert e_{\rho, h}^{n} \Vert_{L^{2}(\Omega)}^{2} 
+ \Vert  e_{u,h}^{n} \Vert_{L^{2}(\Omega)}^{2}\right) 
+ \sum_{n=0}^{m} \tau \Vert  e_{u,h}^{n} \Vert_{H^{1}(\Omega)}^{2} 
\le C (\tau^{2} + h^{3+2\alpha}). 
\end{align}
For$\tau \le \kappa h^{\frac{d}{2}}$ and sufficiently small $\kappa$ and $h$, the inequality above implies 
\begin{align*}
& \Vert  e_{u,h}^{n} \Vert_{L^{2}(\Omega)} \le h^{\frac{3}{2}+\frac{\alpha}{2}}+\tau^{\frac{5}{6}}, \\
& \Vert  e_{u,h}^{n} \Vert_{L^{\infty}(\Omega)} \le C h^{-\frac{d}{2}} 
\Vert  e_{u,h}^{n} \Vert_{L^{2}(\Omega)} \le  h^{-\frac{d}{2}}
(\tau+h^{\frac{3}{2}+\alpha} ) \le 1, \\ 
& \Vert  P_h^{\rm RT} u_{h}^{n}-u^{n}\Vert_{L^{\infty}(\Omega)}  
\le
\Vert  P_h^{\rm RT} (u_{h}^{n}-u^{n})\Vert_{L^{\infty}(\Omega)}  
+ \Vert  P_h^{\rm RT} u^{n}-u^{n}\Vert_{L^{\infty}(\Omega)} \\
&\hspace{96pt}
\le
Ch^{-\frac{d}{2}} \Vert  P_h^{\rm RT} (u_{h}^{n}-u^{n})\Vert_{L^2(\Omega)}  
+ \Vert  P_h^{\rm RT} u^{n}-u^{n}\Vert_{L^{\infty}(\Omega)} \\
&\hspace{96pt} \le Ch^{-\frac{d}{2}}(\tau+h^{\frac32+\alpha}) \\
&\hspace{96pt} \le C\kappa +Ch^{\alpha} 
\\
&\hspace{96pt} \le 2 ,
\\
& \sum_{n=0}^{m} \tau \Vert \nabla  e_{u,h}^{n} \Vert_{L^{2}(\Omega)}^{2} 
\le C (\tau^{2} + h^{3+2\alpha}) 
\le C(\kappa^2h^3 + h^{3+2\alpha} ) \\
&\hspace{87pt} 
\le
(\kappa+h^{\alpha}) h^3 
\quad\mbox{(when $\kappa$ and $h$ are sufficiently small).} 
\end{align*}
This proves (\ref{now_step_error2}, \ref{now_step_error3}, \ref{now_step_error5}, \ref{now_step_error4}). 
Since (\ref{now_step_error1}) has been proved in (\ref{proved_now_step_error1}), 
the mathematical induction is closed. Consequently, the estimates 
(\ref{now_step_errors}) and (\ref{true_error_u_rho}) hold for $m = N$ (with the same constants), which imply the desired estimate in Theorem~\ref{thm_proj_error}.

\section{Numerical experiments}

In this section, we present numerical examples to illustrate the convergence of the numerical method shown 
in Theorem~\ref{thm_proj_error}.  All the computations are performed by Firedrake \cite{Rathgeber2016}. 

In order to test the order of convergence, we consider the following equations with source terms $f$ and $g$: 
\begin{equation}\label{equ:test}
    \begin{aligned}
    \partial_{t} \rho + \nabla\cdot (\rho u) = f 
    &&\mbox{in}\,\,\,\Omega\times(0,T], \\
    \rho \partial_{t} u + \rho  (u \cdot \nabla) u + \nabla p - \mu \Delta u = g
    &&\mbox{in}\,\,\,\Omega\times(0,T], \\ 
    \nabla\cdot u = 0
    &&\mbox{in}\,\,\,\Omega\times(0,T],
    \end{aligned}
\end{equation}
with $T=0.25$ and $\mu = 0.001$. The source terms $f$ and $g$ are constructed by substituting an exact solution to the equations. 
The errors of the numerical solutions with stepsize $\tau$ and mesh size $h$ are denoted by 
\begin{equation}
    \| E_\rho^{\tau,h}\|_{\ell^\infty(L^2)} = \max\limits_{1\le n\le N_T}\|\rho_h^n - \rho^n\|_{L^2}, \quad
    \| E_u^{\tau,h}\|_{\ell^\infty(L^2)}= \max\limits_{1\le n\le N_T}\|u_h^n - u^n\|_{L^2}, \quad
\end{equation}
The convergence order in space is computed by using the following formula, 
\begin{equation}
 \text{convergence order} = \frac{\log(E_\rho^{\tau,h_1}/E_\rho^{\tau,h_2})}{\log(h_1/h_2)}
\end{equation}
with a sufficiently small stepsize $\tau$. 
The convergence order in time is computed by using the following formula, 
\begin{equation}
 \text{convergence order} = \frac{\log(E_\rho^{\tau_1,h}/E_\rho^{\tau_1,h})}{\log(\tau_1/\tau_2)}
\end{equation}
with a sufficiently small mesh size $h$. 

In two dimensions, we consider the problem on the unit square $\Omega = (0, 1)\times (0, 1)$ with the following exact solution:
\begin{equation}
    \left\{
    \begin{aligned}
        & \rho = 2 + x(x-1)\cos(\sin(t)) + y(y-1)\sin(\sin(t)), \\
        & u = \big(\sin^2(\pi x)\sin(2\pi y), -\sin(2\pi x)\sin^2(\pi y)\big), \\
        & p = tx + y - \frac{t + 1}{2}.
    \end{aligned}
    \right.
\end{equation}
We test the convergence order in space by choosing a sufficiently small stepsize $\tau = 1/2048$ so that the error from temporal discretization is negligible in comparison with the error from spatial discretization. 
The errors of the numerical solutions are presented in Table~\ref{table:case2d-space}, where second-order convergence is observed for both $\rho$ and $u$. This is consistent with the theoretical result in Theorem \ref{thm_proj_error}. 

The convergence order in time is computed by choosing $h = \tau^{1/2}$ and presented in Table \ref{table:case2d-time}, where first-order convergence in time is observed. This is also consistent with the theoretical result in Theorem \ref{thm_proj_error}. 

\begin{table}[h!]
    \centering\small
    \caption{Spatial convergence with $\tau= {1/2048}$ }\label{table:case2d-space}
    \begin{tabular}{ccccc}
        \toprule
         $h$
         &  $\| E_\rho^{\tau,h}\|_{\ell^\infty(L^2)}$   &  convergence order 
         &  $\| E_u^{\tau,h}\|_{\ell^\infty(L^2)}$      &  convergence order  \\
         \midrule
         1/8 &  7.48e-06 &    -  &  3.17e-04 &    -  \\
        1/10 &  4.84e-06 &  1.95 &  2.01e-04 &  2.04 \\
        1/12 &  3.42e-06 &  1.91 &  1.39e-04 &  2.02 \\
        1/14 &  2.57e-06 &  1.85 &  1.02e-04 &  2.00 \\
    \bottomrule
\end{tabular}
\end{table}

\begin{table}[h!]
    \centering\small
    \caption{Temporal convergence with $h= \tau^{1/2}$}\label{table:case2d-time}
    \vspace{.05in}
    \begin{tabular}{ccccc}
        \toprule
         $\tau$
         &  $\| E_\rho^{\tau,h}\|_{\ell^\infty(L^2)}$   &  convergence order 
         &  $\| E_u^{\tau,h}\|_{\ell^\infty(L^2)}$      &  convergence order \\
         \midrule
         1/16 &  1.67e-03 &    -  & 3.96e-03 &    -  \\
         1/36 &  7.38e-04 &  1.01 & 1.71e-03 &  1.04 \\
         1/64 &  4.15e-04 &  1.00 & 9.51e-04 &  1.02 \\
        1/100 &  2.65e-04 &  1.00 & 6.05e-04 &  1.01 \\
        \bottomrule
    \end{tabular}
\end{table}

In three dimensions, we consider the problem in a unit cube $\Omega = (0, 1) \times (0, 1) \times (0, 1)$ with the follow exact solution:
\begin{equation}
    \left\{
    \begin{aligned}
        & \rho = 2 + \frac{1}{3}\big(\sin(\pi x) + \sin(\pi y) + \sin(\pi z)\big)\sin(\pi t + \frac\pi2), \\
        & u = \big(\sin^2(\pi x)\sin(2\pi y)\sin(2\pi z), 
                   \sin(2\pi x)\sin^2(\pi y)\sin(2\pi z),
                  -2\sin(2\pi x)\sin(2\pi y)\sin^2(\pi z)
              \big), \\
        & p = t(x + y) + z - \frac{t + 1}{2}.
    \end{aligned}
    \right.
\end{equation}
The errors of the numerical solutions and the convergence orders in space and time are presented in Table \ref{table:case3d-space} and Table \ref{table:case3d-time}, respectively. 
Second-order convergence in space and first-order convergence in time are observed,
which are consistent with the theoretical result in Theorem \ref{thm_proj_error}. 

\begin{table}[h!]
    \centering\small
    \caption{Spatial convergence with $\tau= 1/2048$ }\label{table:case3d-space}
    \begin{tabular}{ccccc}
        \toprule
         $h$
         &  $\| E_\rho^{\tau,h}\|_{\ell^\infty(L^2)}$   &  convergence order 
         &  $\| E_u^{\tau,h}\|_{\ell^\infty(L^2)}$      &  convergence order \\
         \midrule
         1/10 &  4.95e-03 &    -  &  1.00e-01 &    -  \\
         1/12 &  3.51e-03 &  1.89 &  6.98e-02 &  1.99 \\
         1/14 &  2.59e-03 &  1.96 &  5.00e-02 &  2.17 \\
         1/16 &  1.99e-03 &  1.99 &  3.68e-02 &  2.29 \\
        \bottomrule
    \end{tabular}
\end{table}

\begin{table}[h!]
    \centering\small
    \caption{Temporal convergence with $h= \tau^{1/2}$ }\label{table:case3d-time}
    \begin{tabular}{ccccc}
        \toprule
         $\tau$
         &  $\| E_\rho^{\tau,h}\|_{\ell^\infty(L^2)}$   &  convergence order 
         &  $\| E_u^{\tau,h}\|_{\ell^\infty(L^2)}$      &  convergence order \\
         \midrule
        1/256 &  3.45e-03 &    -  &  3.74e-02 &    -  \\
        1/324 &  2.72e-03 &  1.01 &  2.83e-02 &  1.19 \\
        1/400 &  2.20e-03 &  1.01 &  2.19e-02 &  1.21 \\
        1/484 &  1.81e-03 &  1.01 &  1.74e-02 &  1.23 \\
        1/576 &  1.52e-03 &  1.02 &  1.40e-02 &  1.23 \\
        \bottomrule
    \end{tabular}
\end{table}

In the two examples above, the exact solutions are sufficiently smooth. 
Finally, we also consider an exact solution which is not sufficiently smooth, 
\begin{equation}
    \left\{
    \begin{aligned}
        & \rho = 2 + g(x, c)\cos(\sin t) + \big(g(y, c) + g(z, c)\big)\sin(\sin t),  \\
        & u = \big(\sin^2(\pi x)\sin(2\pi y)\sin(2\pi z), 
                   \sin(2\pi x)\sin^2(\pi y)\sin(2\pi z),
                  -2\sin(2\pi x)\sin(2\pi y)\sin^2(\pi z)
              \big), \\
        & p = t(x + y) + z - \frac{t + 1}{2},
    \end{aligned}
    \right.
\end{equation}
where
\begin{equation}
    g(x, c) =  \Big| x - \frac12 \Big|^c
\end{equation}
with $c = 1.51$ on unit cube $\Omega=(0, 1)\times(0, 1)\times(0, 1)$.
This exact solution satisfies $\rho\in H^{2+\alpha}(\Omega)$ for some $\alpha\in(0,0.01)$. 
The errors of the numerical solutions and the convergence orders in space and time are presented in Table~\ref{table:case3d2-space} and Table~\ref{table:case3d2-time}, respectively. 
Again, second-order convergence in space and first-order convergence in time are observed, 
which are consistent with the theoretical result in Theorem \ref{thm_proj_error}. 

\begin{table}[h!]
    \centering\small
    \caption{Spatial convergence with $\tau= {1/2048}$ }\label{table:case3d2-space}
    \begin{tabular}{ccccc}
        \toprule
         $h$
         &  $\| E_\rho^{\tau,h}\|_{\ell^\infty(L^2)}$   &  convergence order 
         &  $\| E_u^{\tau,h}\|_{\ell^\infty(L^2)}$      &  convergence order \\
         \midrule
        1/10 &  3.10e-03 &    -  &  8.99e-02 &    - \\
        1/12 &  2.17e-03 &  1.94 &  6.19e-02 &  2.04\\
        1/14 &  1.59e-03 &  2.04 &  4.40e-02 &  2.21\\
        1/16 &  1.20e-03 &  2.09 &  3.23e-02 &  2.32\\
    \bottomrule
\end{tabular}
\end{table}

\begin{table}[h!]
    \centering\small
    \caption{Temporal convergence with $h= \tau^{1/2}$ }\label{table:case3d2-time}
    \begin{tabular}{ccccc}
        \toprule
         $\tau$
         &  $\| E_\rho^{\tau,h}\|_{\ell^\infty(L^2)}$  &  convergence order 
         &  $\| E_u^{\tau,h}\|_{\ell^\infty(L^2)}$    &  convergence order \\
         \midrule
        1/576 &  4.95e-04 &    -  & 1.25e-02 &    -  \\
        1/676 &  4.18e-04 &  1.06 & 1.03e-02 &  1.19 \\
        1/784 &  3.57e-04 &  1.06 & 8.66e-03 &  1.18 \\
        1/900 &  3.09e-04 &  1.05 & 7.37e-03 &  1.17 \\
       1/1024 &  2.70e-04 &  1.05 & 6.36e-03 &  1.15 \\
    \bottomrule
\end{tabular}
\end{table}

\section{Conclusions}

We have present error analysis for a fully discrete, linearized semi-implicit and decoupled FEM for the coupled system \eqref{DNS_eqs} describing incompressible flow with variable density. Compared to the previous work in \cite{Cai-Li-Li-2020}, the error analysis in this paper is obtained in three dimensions under more realistic $H^{2+\alpha}$ regularity assumptions on the solution in a convex polyhedron.  
In the numerical method for the velocity equation (2.16b), we have added a stabilization term   
  $$
  \frac{1}{2} (  D_{\tau} \chi ( \rho_{h}^{n}) \,u_{h}^{n},  v_{h}) 
- \frac{1}{2}( \chi (\rho_{h}^{n})u_{h}^{n-1}, \nabla (u_{h}^{n} \cdot  v_{h})) ,
  $$
which helps to stabilize the velocity equation and therefore yields the energy inequality \eqref{energy-inequality} unconditionally, which holds also for small viscosity $\mu$. 
Since our error analysis strongly relies on the viscosity in the momentum equations, 
we have not considered the convection dominate case in this paper. 
But the energy inequality \eqref{energy-inequality} implies that the method at least maintains the energy stability of the numerical solution in the convection dominant case. 
The error analysis for system \eqref{DNS_eqs} in the convection dominant case is more challenging and remains open.

\section{Declarations}
\noindent {\bf Funding:}
The work of B. Li and Z. Yang was supported in part by National Natural Science Foundation of China (NSFC grant 12071020) 
and an internal grant of The Hong Kong Polytechnic University (Project 4-ZZKQ).
Weifeng Qiu is supported by a grant from the Research Grants Council of the Hong Kong Special Administrative Region, China 
(Project No. CityU 11302718).  \\

\noindent {\bf The Conflict of Interest Statement:}
No conflict of interest exists.\\

\noindent {\bf Availability of data and material:} 
Not applicable.\\

\noindent {\bf Code availability:}
Not applicable.\\

\noindent {\bf Authors' contributions:}
Buyang Li,  Weifeng Qiu and Zongze Yang have participated sufficiently in the work to take public
responsibility for the content, including participation in the concept, method,
analysis and writing. All authors certify that this material or similar material
has not been and will not be submitted to or published in any other publication.



\end{document}